\documentclass[12pt,a4paper,oneside]{amsart}
\usepackage{amsmath,amssymb,amsthm,graphicx,mathrsfs,url}
\usepackage[usenames,dvipsnames]{color}
\usepackage[colorlinks=true,linkcolor=Red,citecolor=Green]{hyperref}
\usepackage{amsxtra}
\usepackage{wasysym} 
\usepackage{graphicx}
\usepackage{subcaption}

\usepackage{array}
\newcolumntype{P}[1]{>{\centering\arraybackslash}m{#1}}

\usepackage{blindtext}
\usepackage{stmaryrd}

\def\?[#1]{\textbf{[#1]}\marginpar{\Large{\textbf{??}}}}

\let\epsilon=\varepsilon 

\newtheorem{theo}{Theorem}
\newtheorem*{theo*}{Theorem}
\newtheorem{prop}{Proposition}[section]
\newtheorem*{prop*}{Proposition}	
\newtheorem{defi}[prop]{Definition}
\newtheorem*{defi*}{Definition}

\newtheorem{lemm}[prop]{Lemma}
\newtheorem*{lemm*}{Lemma}
\newtheorem{corr}[prop]{Corollary}
\newtheorem{rem}{Remark}
\numberwithin{equation}{section}

\usepackage{biblatex}
\DeclareMathOperator{\tr}{tr}

\usepackage{scalerel}

\newcommand\reallywidehat[1]{\arraycolsep=0pt\relax%
\begin{array}{c}
\stretchto{
  \scaleto{
    \scalerel*[\widthof{\ensuremath{#1}}]{\kern-.5pt\bigwedge\kern-.5pt}
    {\rule[-\textheight/2]{1ex}{\textheight}} 
  }{\textheight} %
}{0.5ex}\\           
#1\\                 
\rule{-1ex}{0ex}
\end{array}
}

\author{Tristan Humbert}
\email{humbertt@imj-prg.fr}
\address{Sorbonne université, Paris France 75005.}

\begin{document}
\begin{abstract}
For Anosov diﬀeomorphisms on the $3$-torus $\mathbb T^3$ which are strongly partially hyperbolic with expanding center, we construct systems of strong unstable and center stable Margulis measures which are holonomy-invariant. This allows us to obtain a characterization of the measures of maximal unstable entropy in terms of their conditional measures along strong unstable leaves. Moreover, we show that the Margulis systems identify with Pollicott-Ruelle (co)-resonant states for the action of the diffeomorphism on $2$-forms. This shows that the unstable topological entropy and a measure of maximal unstable entropy can be retrieved from the spectral approach.
\end{abstract}
\title{Unstable entropy for Anosov diffeomorphisms on the $3$-torus}
\maketitle
\section{Introduction}
\subsection{Anosov diffeomorphisms with a partially hyperbolic splitting}
Let $f\in \mathrm{Diffeo}^{\infty}(\mathbb T^3)$ be a smooth diffeomorphism of $\mathbb T^3$. Suppose that $f$ is  \emph{Anosov} with splitting $T(\mathbb T^3)=E_s\oplus E_{cu}$ where $\mathrm{dim}(E_s)=1$ and $\mathrm{dim}(E_{cu})=2$. Suppose in addition that the unstable bundle $E_{cu}$ of $f$ splits as $E_{cu}=E_c\oplus E_u$ where $E_c$ is dominated by $E_u$, see \S \ref{preli} for the precise definitions. In other words, $f$ is a partially hyperbolic diffeomorphism with splitting $T(\mathbb T^3)=E_u\oplus E_c\oplus E_s$. We will call $E_c$ the \emph{center direction} and $E_u$ the \emph{strong unstable direction} of $f$. We further suppose that $f$ is transitive and orientation preserving\footnote{This assumption is made to simplify the exposition. If $\mathrm{det}(f)<0$, one can consider $f\circ f$ instead. }. Let $\mathcal A_+^{\infty}(\mathbb T^3)$ be the set of such diffeomorphisms. 
\subsection{Unstable entropy}
For a partially hyperbolic diffeomorphism $f$ and an invariant probability measure $\mu$, Hu, Hua and Wu \cite{HuHuWu} defined the notion of \emph{metric entropy of $\mu$ with respect to the strong unstable foliation $E_u$,} which we denote by $h^u_{\mu}(f)$. Another important invariant associated to $E_u$ is the \emph{unstable topological entropy} $h^u_{\mathrm{top}}(f)$; an invariant introduced by Saghin and Xia \cite{SaXi} which captures the exponential growth rate of the volume of balls in the strong unstable manifold. The two notions of entropies are linked by a \emph{variational principle} \cite[Theorem D]{HuHuWu}, 
\begin{equation}
\label{eq:varia11}
h^u_{\mathrm{top}}(f)=\sup_{\mu \in \mathcal P_f(\mathbb T^3)}h^u_\mu(f),
\end{equation}
where $\mathcal P_f(\mathbb T^3)$ denotes the set of $f$-invariant probability measures on $\mathbb T^3$, see \S \ref{secUnstable} for a quick review of these notions.
 A probability measure which attains the supremum is called a \emph{measure of maximal $u$-entropy} or $u$-MME. Their set will be denoted by $\mathcal{M}^u(f)$.  Hu,Wu and Zhu \cite[Theorem B]{HuWuZh} showed that $\mathcal{M}^u(f)$ is always non-empty, convex, compact and that its extreme points are precisely the ergodic measures of maximal $u$-entropy. In this paper, we compute  $h^u_{\mathrm{top}}(f)$ and study some properties of $\mathcal{M}^u(f)$. To state our results, let us define the \emph{center Jacobian},
$$
J^c_f\in C^{\alpha}(\mathbb T^3,\mathbb R),\ \forall x\in \mathbb T^3, \quad J_f^c(x):=-\ln \|df_x|_{E_c(x)}\|.
$$
Since $J^c_f\in C^\alpha(\mathbb T^3)$ and $f$ is Anosov there exists a unique associated equilibrium state $\mu_{J^c_f}$. Let $P(J^c_f)$ be the pressure of $J^c_f$. Recall that $f$ is conjugated to its action on homology $A\in \mathrm{GL}_3(\mathbb Z)$ via a Hölder continuous homeomorphism. Since, $f\in \mathcal A_+^{\infty}(\mathbb T^3)$, the spectrum of $A$ is given by $\mathrm{Spec}(A):=\{\lambda_s<1<\lambda_c<\lambda_u\}$. With these notations, our first result reads:
\begin{theo}[Unstable entropy]
\label{main theo}
For any $f\in \mathcal A_+^{\infty}(\mathbb T^3)$, one has
\begin{equation}
\label{eq:equality}
e^{h^u_{\mathrm{top}}(f)}=\lambda_u=e^{P(J^c_f)}, \quad \mu_{J^c_f}\in \mathcal{M}^u(f).
\end{equation}
As a consequence, $\mathcal A_+^{\infty}(\mathbb T^3)\ni f\mapsto h^u_{\mathrm{top}}(f)$ is invariant by topological conjugacy and thus locally constant.
\end{theo}
In particular, as an equilibrium state for an Anosov map, $\mu_{J^c_f}$ satisfies exponential mixing. Moreover, $\mu_{J^c_f}$ can be expressed via Bowen's formula in terms of periodic points, see for instance \cite[Theorem 20.3.7]{HK}.

In a recent survey \cite{Ta}, Tahzibi asked whether for an Anosov diffeomorphism which is partially hyperbolic with expanding center, a measure of maximal $u$-entropy is necessarily a measure of maximal entropy. We will say that $E_s$ and $E_u$ are \emph{jointly integrable} if there is a 2 dimensional invariant foliation tangent to $E_s\oplus E_u$. Using a criteria due to Gan and Shi \cite{GaSh}, we show the following result.
\begin{corr}
\label{cor1}
Let $f\in\mathcal A_+^{\infty}(\mathbb T^3)$. Then $\mu_{J^c_f}$ is a measure of maximal entropy if and only if $E_u$ and $E_s$ are jointly integrable, if and only if the center Lyapunov exponent is equal to $\lambda_c$ for all periodic orbits.
\end{corr}
For $f\in \mathcal A_+^{\infty}(\mathbb T^3)$, joint integrability of $E_u$ and $E_s$ is also equivalent to $f$ being not \emph{accessible}, see  \cite{RGZ,GaSh}. Accessibilty is $C^r$-dense ($r\geq1$) \cite{DW,BRHRHTU} and $C^1$-open \cite{Di} in $ \mathcal A_+^{\infty}(\mathbb T^3)$. We thus answer negatively to \cite[Question 2]{Ta} for an open and dense set of diffeomorphisms $f\in \mathcal A_+^{\infty}(\mathbb T^3)$.
\subsection{Margulis system of measures}
Next, we study more precisely the structure of $\mathcal M^u(f)$. To this end, we will construct systems of measures $(\mu^u_x)_{x\in \mathbb T^3}$ (resp. $(\mu^{cs}_x)_{x\in \mathbb T^3}$) on the strong unstable (resp. center stable) manifolds. The study of measures of maximal entropy (and more generally equilibrium states) via system of leaf measures has a rich and long history, see the introductions of \cite{CPZ20,Hum} for an account of the existing literature.

A \emph{system of $u$-measures} $($resp. \emph{$cs$-measures$)$} is a family of Borelian measures $(\mu^u_x)_{x\in \mathbb T^3}$ (resp. $(\mu^{cs}_x)_{x\in \mathbb T^3})$ on strong unstable manifolds $(\mathcal W^u(x))_{x\in \mathbb T^3}$ (resp. center stable manifolds $(\mathcal W^{cs}(x))_{x\in \mathbb T^3}$), see Definition \ref{def}. For $p,q\in \mathbb T^3$, close enough such that $p\in \mathcal W^u (q)$ (resp. $p\in \mathcal W^{cs} (q)$), we denote by $\mathrm{Hol}^u_{p,q}:\mathcal W^{cs}(p)\to \mathcal W^{cs}(q)$ (resp. $\mathrm{Hol}^{cs}_{p,q}:\mathcal W^{u}(p)\to \mathcal W^{u}(q)$) the local $u$-holonomy (resp. $cs$-holonomy), see \S \ref{secholo}.
\begin{theo}[Margulis system]
\label{theoMarg}
Let $f\in \mathcal A_+^{\infty}(\mathbb T^3)$. There exists a system of $u$-measures $(\mu^{u}_x)_{x\in \mathbb T^3}$ and a system of $cs$-measures $(\mu^{cs}_x)_{x\in \mathbb T^3}$ such that:
\begin{itemize}
\item  the systems are conformal with respect to $f$
\begin{equation}
\label{eq:conf}
\forall x\in \mathbb T^3, \quad f^*\mu^u_x=e^{h^u_{\mathrm{top}}(f)}\mu^u_{f^{-1}(x)},\quad  f_*\mu^{cs}_x=e^{h^u_{\mathrm{top}}(f)}\mu^{cs}_{f(x)};
\end{equation} 
\item the systems are holonomy invariant:
\begin{equation}
\label{eq:holhol}
(\mathrm{Hol}^{cs}_{x,z})_*\mu^u_z= \mu^u_x,\quad (\mathrm{Hol}^{u}_{y,w})_*\mu^{cs}_w= \mu^{cs}_y,
\end{equation}
for any points $x,z$ $($resp. $y,w)$ for which the center stable $($resp. strong unstable$)$ holonomy is well defined.
\end{itemize}
The systems are called $u$ and $cs$-Margulis systems of measures respectively.
Moreover, the measure $\mu_{J^c_f}$ has a strong unstable/center stable local product structure with $u$ $($resp. $cs)$ conditionals given by $(\mu^u_x)_{x\in \mathbb T^3}$ $($resp. $(\mu^{cs}_x)_{x\in \mathbb T^3})$, see \eqref{eq:localprod}.
\end{theo}
Theorem \ref{theoMarg} should be compared to the construction of the measure of maximal entropy via systems of leaf measures. In the hyperbolic case, this is due to Sinai for maps\footnote{Note that Sinai constructs a $cu$-system and a $s$-system. Our construction is thus closed to Margulis' construction.} \cite{Si} and to Margulis for flows \cite{Marg}. Recently, different works have extended the construction of leaf measures to some partially hyperbolic settings, see \S \ref{seccomp} for a comparison with existing works. The Margulis system $(\mu^u_x)_{x\in \mathbb T^3}$ can be used to characterize the measures of maximal unstable entropy. 
\begin{corr}
\label{cor}
Let $f\in \mathcal A_+^{\infty}(\mathbb T^3)$ and $\nu\in \mathcal P_f(\mathbb T^3)$. Then $\nu\in \mathcal M^u(f)$ if and only if the conditionals of $\nu$ along $\mathcal W^u$ are given by $(\mu^u_x)_{x\in \mathbb T^3}$ $($up to a constant rescaling$)$, $\nu$-a.e.
\end{corr}
The author does not know if $\mu_{J^c_f}$ is the unique measure of maximal $u$-entropy or not.
We note that because of the center direction being expanding, the usual Hopf argument does not seem to give the uniqueness as in \cite{CPZ20,BFT,UVYY}.

Another important class of invariant measures is given by \emph{$u$-Gibbs measures}. They correspond to invariant measures whose conditionals along strong stable manifolds are absolutely continuous with respect to the Lebesgue measure or equivalently to $u$-equilibrium states associated to the \emph{strong unstable Jacobian}, see \cite[Theorem C]{HuWuZh} for more details. For $f\in \mathcal A^\infty_+(\mathbb T^3)$ which are close to volume preserving, the uniqueness of the $u$-Gibbs measure when $E_u$ and $E_s$ are not jointly integrable and $\mathcal W^u$ is minimal   was proved by Alvarez, Leguil, Obata and Santiago in \cite[Theorem A]{ALOS}. The minimality of $\mathcal W^u$ for any $f\in \mathcal A^\infty_+(\mathbb T^3)$ was announced by Avila-Crovisier-Eskin-Potrie-Wilkinson-Zhang and we refer to the introduction of \cite{ALOS} for a more detailed exposition of the existing literature on $u$-Gibbs measures.
\subsection{Pollicott-Ruelle resonances}
The initial motivation of this work was to study the \emph{Pollicott-Ruelle spectrum} of $f\in \mathcal A_+^{\infty}(\mathbb T^3)$. For $0\leq k \leq 3$, let $\mathcal D'_{k}(\mathbb T^3)$ denote the space of $k$-currents, see \S \ref{secRuelle} for a definition. Let us define 
\begin{equation}
\label{eq:dual1}
T^*(\mathbb T^3)=E_s^*\oplus E_{cu}^*,\quad E_s^*(E_s)=0, \ E_{cu}^*(E_{cu})=0.
\end{equation}
One can associate to an Anosov map $f$ a discrete spectrum $\mathrm{Res}_k(f)\subset \mathbb C$ : the \emph{Pollicott-Ruelle resonances}. Recall that $\lambda \in \mathrm{Res}_k(f)$ if and only if
\begin{equation}
\label{eq:res}
\exists u\in \mathcal D'_k(\mathbb T^3)\setminus\{0\},\quad f_*u=\lambda u, \ \mathrm{WF}(u)\subset E_{cu}^*. 
\end{equation}
Here, $\mathrm{WF}(u)\subset T^*\mathbb T^3\setminus \{0\}$ denotes the wavefront set of $u$, see \cite[Chapter VIII]{Hor}. In this case, the current $u$ is called a \emph{resonant state} associated to the resonance $\lambda$. We have a dual notion of \emph{co-resonant state}:
\begin{equation}
\label{eq:cores}
\exists u\in \mathcal D'_{3-k}(\mathbb T^3)\setminus\{0\},\quad f^*u=\lambda u, \ \mathrm{WF}(u)\subset E_{s}^*. 
\end{equation}
 We say that $\lambda \in \mathbb R_+$ is the \emph{first resonance for the action on $k$-forms} if $\lambda\in \mathrm{Res}_k(f)$ and $\mathrm{Res}_k(f)\subset \{z\in \mathbb C, |z|\leq \lambda\}.$ For $k=0$ (resp. $k=1$), it is known that the first resonance is given by $1$ (resp. $e^{h_{\mathrm{top}}(f)}$) and that the associated (co)-resonant states can be used to reconstruct the SRB measure (resp. measure of maximal entropy), see \cite{Ba,GouLiv} for Anosov maps and \cite{BC,Hum} for Anosov flows. For $f\in \mathcal A_+^{\infty}(\mathbb T^3)$, it is therefore natural to ask what is the first resonance for the action on $2$-forms and whether the corresponding (co)-resonant states have dynamical meaning or not.
\begin{theo}[First Pollicott-Ruelle resonance]
\label{theores}
Let $f\in \mathcal A_+^{\infty}(\mathbb T^3)$. Then $e^{h^u_{\mathrm{top}}(f)}$ is the first Pollicott-Ruelle resonance for the action on $2$-forms. Moreover, $e^{h^u_{\mathrm{top}}(f)}$ is the only resonance on the critical circle $\{ z\in \mathbb C, |z|=e^{h^u_{\mathrm{top}}(f)}\}$, it does not have Jordan block and the space of resonant $($resp. co-resonant$)$ states is one-dimensional.

Let $\theta$ $($resp. $\nu)$ be a corresponding resonant $($resp. co-resonant$)$ state associated to $  e^{h^u_{\mathrm{top}}(f)}$. Then one has
\begin{equation}
\label{eq:restric}
\forall x\in \mathbb T^3, \quad \theta_{|\mathcal W^{cs}(x)}=\mu^{cs}_x, \ \nu_{|\mathcal W^{u}(x)}=\mu^{u}_x,
\end{equation}
where $g|_{\mathcal N}$ denotes the restriction of a current $g$ to a submanifold $\mathcal N.$\footnote{see for instance \cite[Corollary 8.2.7]{Hor}.}
Finally, the trace of the spectral projector at the resonance $\lambda=e^{h^u_{\mathrm{top}}(f)}$ is given by $\mu_{J^c_f}$.
\end{theo}
In other words, one can recover the unstable topological entropy as a Pollicott-Ruelle resonance. The (co)-resonant states identify with Margulis systems of  measures and $\mu_{J^c_f}$ can be reconstructed as the trace of the spectral projector at the first resonance.
\subsection{Comparison with existing results.}
\label{seccomp}
\subsubsection{Unstable entropy} The topological unstable entropy was shown to be upper semi-continuous by Yang \cite{Ya} and continuous in our setting by Wu \cite{Wu}. That $h^u_{\mathrm{top}}(f)$ is constant in a neighborhood of a linear automorphism $A$ follows from the work of Hua, Saghin and Xia \cite{HSX}. Here, we extend this observation to all Anosov maps $f$ isotopic to $A$. The facts that the unstable topological entropy coincides with the pressure of $J^c_f$ and that $\mu_{J^c_f}\in \mathcal{M}^u(f)$ seem to be new, to the best of the author's knowledge.
\subsubsection{Systems of Margulis measures}

Recently, different works have constructed systems of leaf measures for some classes of partially hyperbolic systems in order to study measures of maximal entropy, and more generally equilibrium states. Most relevant to our paper are the constructions of Climenhaga, Pesin and Zelerowicz \cite{CPZ20}, Parmenter and Pollicott \cite{PP}, Carrasco and Rodriguez-Hertz \cite{CH2}, Buzzi, Fisher and Tazhibi \cite{BFT} and Ures, Viana, Yang and Yang \cite{UVYY}.

We emphasize that Theorem \ref{theoMarg} does not follow from any of the previously cited works. Indeed, for $f\in \mathcal A_+^{\infty}(\mathbb T^3)$, since the center direction $E_c$ is uniformly expanded, $f$ does not satisfy the \emph{Lyapunov stability} condition of \cite{CPZ20,PP}, it is not a \emph{center isometry} as required in \cite{CH2}, nor it is a \emph{flow type diffeomorphism} as in \cite{BFT}. Although any $f\in \mathcal A_+^{\infty}(\mathbb T^3)$ is topologically conjugated to an Anosov automorphism $A$ of $\mathbb T^3$, it does not necessarily \emph{factor over Anosov} as in \cite{UVYY}. Indeed, as shown in \cite{RGZ}, the conjugacy maps the strong unstable foliation of $f$ to that of $A$ if and only if $E_u$ and $E_s$ are jointly integrable. As stated in the introduction, this condition fails for a dense and open subset of $ \mathcal A_+^{\infty}(\mathbb T^3)$. For such $f$, the Margulis system $(\mu^u_x)_{x\in \mathbb T^3}$ cannot be obtained by simply pulling back by the conjugacy the Lebesgue measure of $\mathcal W^u_A$, where $\mathcal W^u_A$ is the strong unstable manifold of $A$. When $f$ factor over Anosov, Corollary \ref{cor} does follow from \cite[Theorem B]{UVYY} but since $f$ does not have \emph{c-mostly contracting center}, \cite[Theorem C]{UVYY} does not apply and the uniqueness of the $u$-MME is open even in this case. For a linear automorphism $A\in \mathcal A_+^{\infty}(\mathbb T^3)$, the foliations are linear and are in particular absolutely continuous. This means that using Corollary \ref{cor} and the standard Hopf argument, one gets that $A$ has a unique $u$-MME, see for instance \cite[Proposition 3.8]{BFT}.

To the best of the author's knowledge, Theorem \ref{theoMarg} is the first general construction of Margulis systems for a class of partially hyperbolic systems with uniformly expanding center. From a technical point of view, this makes showing the $\mathrm{Hol}^{cs}$-invariance \eqref{eq:holhol} of $(\mu^u_x)_{x\in \mathbb T^3}$ much harder due to the lack of regularity of center stable holonomies. To overcome this difficulty, we combine dynamical ideas inspired by \cite{BFT,CH2,CPZ20} with functional techniques as in \cite{Hum,Hum2}, see the next subsection for an outline of the argument. We note that Alvarez, Leguil, Obata and Santiago studied $u$-Gibbs measures for $f\in \mathcal A_+^{\infty}(\mathbb T^3)$\footnote{Their techniques actually allow them to study $C^2$ diffeomorphisms rather than smooth ones.} in \cite{ALOS}. The lack of regularity of center stable holonomies is also aknowledged as a source of technical difficulty in their work, see \cite[\S 2.2.3]{ALOS}, although they use different techniques to overcome it.

\subsubsection{Pollicott-Ruelle resonances} Theorem \ref{theores} seems to be the first result on the first resonance for Anosov maps for $k\neq 0,d_s$ where $d_s=\mathrm{dim}(E_s)$. The close link between  (co)-resonant states for $k=0,d_s$ and systems of leaf measures for Anosov diffeomorphism is well understood \cite{GouLiv}. This was used by the author in \cite{Hum} to reprove classical facts on the thermodynamical formalism for Anosov flows. 

Recently, the author combined the functional approach for general Anosov actions obtained by Guedes Bonthonneau, Guillarmou, Hilgert, and Weich \cite{GBGHW,GBGW} and the work of Carrasco and Rodriguez-Hertz \cite{CH2} to study the measure of maximal entropy for minimal Anosov actions \cite{Hum2}. The present work can be seen as another instance of partially hyperbolic systems for which certain important dynamical invariant and measure can be reconstructed from the spectral approach.

We note that the functional approach was also used to construct physical and SRB measures for some partially hyperbolic systems, see \cite{CaLi} and the references therein.
\subsection{Structure of the paper}
In \S \ref{preli}, we define the class $\mathcal A_+^{\infty}(\mathbb T^3)$ and recall some facts about the dynamical foliations tangent to the invariant bundles $E_\bullet$, for $\bullet=s,c,u,cs,cu$ and the local product structure they define. We recall as well the definition of the conditional entropy of an invariant measure along the strong unstable foliation and prove a basic lemma about the pressure of the center Jacobian.

In \S \ref{sec3}, we construct the unstable system of Margulis measures $(\mu^u_x)_{x\in \mathbb T^3}$ using a compactness argument inspired by \cite{BFT,CH2}. The compactness used in the argument follows from a precise volume bound \eqref{eq:lowerupper} which is a consequence of the work of Potrie \cite{Po}. We emphasize that since the complementary foliation $\mathcal W^{cs}$ is not contracting, this construction does not suffice to obtain the holonomy invariance \eqref{eq:holhol} of $(\mu^u_x)_{x\in \mathbb T^3}$ as in \cite{BFT,CH2}. Nevertheless, using a Carathéodory construction inspired by Climenhaga, Pesin and Zelerowicz \cite{CPZ20}, we show that $(\mu^u_x)_{x\in \mathbb T^3}$ is measurable which allows us to deduce Corollary \ref{cor} using an argument of Tahzibi \cite{Ta}.

In \S \ref{sec4}, we recall some facts about Pollicott-Ruelle resonances and show Theorem \ref{theores}, except for \eqref{eq:restric}, using the functional approach. The argument is shorter than in \cite{Hum,Hum2} since $f\in \mathcal A_+^{\infty}(\mathbb T^3)$ is conjugated to a linear model, which we can use to simplify some proofs. We also deduce Theorem \ref{main theo} and Corollary \ref{cor1}.

Finally, \S \ref{sec5} is devoted to the construction of the $cs$-system and the proofs of \eqref{eq:holhol} and \eqref{eq:restric}. We use a combination of functional and dynamical ideas. In summary:
\begin{itemize}
\item the $cs$-system $(\mu^{cs}_x)_{x\in \mathbb T^3}$ is constructed from the resonant state using the functional approach, see Proposition \ref{propdur}. This establishes \eqref{eq:holhol} and \eqref{eq:restric} for $(\mu^{cs}_x)_{x\in \mathbb T^3}$;
\item we show using the functional approach that $\mu_{J^c_f}$ can be obtained as the limit of pullbacks of $(\mu^{cs}_x)_{x\in \mathbb T^3}$, see Proposition \ref{prophol}. This should be compared with \cite[Theorem 4.7 (1)]{CPZ20} and \cite[Theorem 1.1]{PP} although our techniques are quite different;
\item we deduce that $(\mu^{cs}_x)_{x\in \mathbb T^3}$ is equal to the $cs$-conditionals of $\mu_{J^c_f}$ mimicking the dynamical argument of \cite[Lemma 8.1]{CPZ20};
\item using an adaptation of \cite[Lemma 2.12]{CPZ20}, the \emph{local product structure} of $\mu_{J^c_f}$ and Corollary \ref{cor}, we deduce the holonomy invariance \eqref{eq:holhol} of $(\mu^{u}_x)_{x\in \mathbb T^3}$ from the holonomy invariance of $(\mu^{cs}_x)_{x\in \mathbb T^3}$. This last argument seems to be the main novelty compared to previous works and might be of independent interest;
\item finally, we show \eqref{eq:restric} for $(\mu^{u}_x)_{x\in \mathbb T^3}$ by adapting \cite[Lemma 3.2]{Hum}.
\end{itemize}
\textbf{Acknowledgements.} The author would like to first thank Colin Guillarmou and Thibault Lefeuvre for their guidance and advice during the writing of this paper. The author would also like to thank Jérôme Buzzi, Sylvain Crovisier, François LeDrappier, Ali Tahzibi and Amie Wilkinson for answering some questions related to the project.

The author would like to further thank Ali Tahzibi for noticing that the previous argument given in Proposition \ref{prop:last} was incomplete and Jérôme Buzzi for making him aware of a gap in the previous proof of the measurability of the Margulis system in Proposition \ref{leafmeasure}. These issues  have now been resolved in the current version of the paper. He would like to thank Ali Tahzibi for pointing out to him that the $u$-MME is unique for linear automorphisms and Gabriel Rivière for pointing out the work \cite{RS}, which lead to Remark \ref{rem:Gabriel}. Finally, the author would like to thank the anonymous referee for their
careful and thorough review of the paper.

This research was supported by the European Research Council (ERC) under
the European Union’s Horizon 2020 research and innovation programme (Grant agreement no. 101162990 — ADG).

\section{Preliminaries}
\label{preli}
\subsection{Anosov diffeomorphism with a partially hyperbolic splitting}
\subsubsection{Definition} 
Let $f\in \mathrm{Diffeo}^{\infty}(\mathbb T^3)$ be a smooth diffeomorphism. We say that $f$ is \emph{Anosov and partially hyperbolic with expanding center} if there exists a continuous splitting of the tangent bundle
$T(\mathbb T^3)=E_u\oplus E_c\oplus E_s,$
which is $df$-invariant and a Riemannian metric $\|.\|$ \emph{adapted to the splitting} such that
$$x\mapsto \lambda_x^\bullet:=\|df(x)|_{E_\bullet}\|,\quad \bullet=s,c,u, $$
are continuous and satisfy $\lambda_x^s<1<\lambda_x^c<\lambda_x^u$ for any $x\in \mathbb T^3$. In particular, $f$ is an Anosov map with stable bundle $E_s$ and unstable bundle $E_{cu}:=E_c\oplus E_u$. The bundles $E_c$ and $E_u$ are respectively called the center and strong unstable bundle of $f$. Finally, $E_{cs}:=E_c\oplus E_s$ is the center stable bundle of $f$. For a detailed introduction to Anosov diffeomorphisms, we refer to \cite{HK}. 

We will denote by $\mathcal A_+^{\infty}(\mathbb T^3)$ the set of such diffeomorphisms which furthermore preserve the orientation of $\mathbb T^3$. It is easily seen to be $C^1$-open and examples are provided by $A\in \mathrm{GL}_3(\mathbb Z)$ with real distincts eigenvalues $\lambda_s<1<\lambda_c<\lambda_u$.
\subsubsection{Invariant manifolds} Let $f\in \mathcal A_+^{\infty}(\mathbb T^3)$. Since $f$ is Anosov, $E_s$ and $E_{cu}$ are integrable to $f$-invariant foliations $\mathcal W^s$ and $\mathcal W^{cu}$ called respectively the stable and unstable foliations of $f$, see \cite[Theorem 6.2.3]{HK}. Moreover, since $E_u$ is the strong unstable foliation of a partially hyperbolic diffeomorphism, it integrates to a $f$-invariant foliation $\mathcal W^u$ called the strong unstable foliation of $f$, see \cite[Theorem 4.1]{Pe}. Actually, by the work of Brin, Burago and Ivanov \cite{BBI} and Potrie \cite{Po}, any $f\in \mathcal A_+^{\infty}(\mathbb T^3)$ is \emph{dynamically coherent}. This means that the center stable bundle $E_{cs}$ also integrates to a $f$-invariant foliation $\mathcal W^{cs}$ called the center stable foliation of $f$.  The collection of leaves $\mathcal W^c(x):=\mathcal W^{cs}(x)\cap \mathcal W^{cu}(x)$ for $x\in \mathbb T^3$ defines the center foliation which integrates the center bundle $E_c$. For any $\bullet=s,c,u,cs,cu$, any $x\in \mathbb T^3$ and any $\delta>0$, let $\mathcal W^\bullet(x,\delta):=\{y\in \mathcal W^\bullet(x) \mid d^\bullet(x,y)<\delta\}$, where $d^\bullet$ is the distance on the leaf.
\subsubsection{Local product structure and rectangles}
\label{sec:delta}
By transversality of the foliations $\mathcal W^{cs}$ and $\mathcal W^u$, there is an $\epsilon_0>0$ and a $C>0$ such that 
\begin{equation}
\label{eq:Bowen}\forall x,y\in \mathbb T^3, \ d(x,y)<\epsilon<\epsilon_0  \ \Rightarrow \ \mathcal W^{cs}_{C\epsilon}(x)\cap \mathcal W^u_{C\epsilon}(y)=\{[x,y]\},
\end{equation}
where $[x,y]$ is the \emph{Bowen bracket} of $x$ and $y$.
\begin{defi}
A closed set $R\subset \mathbb T^3$ is called a \emph{rectangle} if for every $x,y\in R$, the Bowen bracket $[x,y]$ exists and $[x,y]\in R$.
\end{defi}
Rectangles exist because of the local product structure, see \cite[Eq. $(2.2)$]{CPZ20}.
\subsubsection{Holonomies}
\label{secholo}
We now define the center stable and unstable holonomies. Let $R$ be a rectangle and $x\in  R$. We will denote by 
$\mathcal W^{\bullet}_R(x),$ for $ \bullet =cs,u,$
the connected component of $\mathcal W^{\bullet}(x)\cap R$ containing $x$.
For $x,y\in R$, the \emph{strong unstable holonomy} is 
\begin{equation}
\label{eq:holu}
\mathrm{Hol}^u_{x,y}: \mathcal W^{cs}_R(x)\to \mathcal W^{cs}_R(y),\quad \mathrm{Hol}^u_{x,y}(z):=[y,z]\in \mathcal W^{cs}(y)\cap \mathcal W^u(z).
\end{equation}
The \emph{center stable holonomy} is 
\begin{equation}
\label{eq:holcs}
\mathrm{Hol}^{cs}_{x,y}: \mathcal W^{u}_R(x)\to \mathcal W^{u}_R(y),\quad \mathrm{Hol}^{cs}_{x,y}(z):=[z,y]\in \mathcal W^{cs}(z)\cap \mathcal W^u(y).
\end{equation}
Given $x,y\in R$ and $A\subset \mathcal W^{cs}_R(x), B\subset \mathcal W^{cs}_R(y)$, we say that $A$ and $B$ are \emph{$\mathrm{Hol}^u$-equivalent} if $\mathrm{Hol}^u_{x,y}(A)=B$. Given $\delta>0$, we say that $A$ and $B$ are {$\delta$-equivalent} if the size of the holonomy of is bounded by $\delta$, i.e. if any point in $a\in A$ is  at $d^u$-distance at most $\delta$ from $\mathrm{Hol}^u_{x,y}(a)\in B$. We define likewise the notion of \emph{$\mathrm{Hol}^{cs}$-equivalence}. Two functions $\psi\in C^\infty(\mathcal W^{cs}_R(x))$ and $\phi\in C^\infty(\mathcal W^{cs}_R(y))$ are said to be $\mathrm{Hol}^u$-equivalent (resp. $\delta$-equivalent) if their supports are $\mathrm{Hol}^u$-equivalent (resp. $\delta$-equivalent) and $\phi\circ\mathrm{Hol}^u_{x,y}=\psi$ on the support of $\psi$. We define likewise the notions of $\mathrm{Hol}^{cs}$-equivalence (resp. $\delta$-equivalence) for functions on strong unstable leaves.

The strong unstable foliation is \emph{absolutely continuous}, see \eqref{eq:abs} for the exact statement we will need. This fact will be used in the proof of Proposition \ref{propholu} to show the $\mathrm{Hol}^u$-invariance of $(\mu^{cs}_x)_{x\in \mathbb T^3}$. The center stable foliation is not always absolutely continuous and this explains why proving the $\mathrm{Hol}^{cs}$-invariance of $(\mu^{u}_x)_{x\in \mathbb T^3}$ is harder. We refer to \cite[\S 2.2.3]{ALOS} for a more detailed discussion of the lack of regularity of center stable holonomies.
\subsection{Unstable entropy}
\label{secUnstable}
\subsubsection{Conditional measures} We recall some facts about measurable
partitions and conditional measures and refer to \cite{Roh} for more details. Let $(X,\mu)$ be a probability space. For $\xi$ a partition of $X$ into $\mu$-measurable sets, we write $\xi(x)\in \xi$ for the element of the partition containing $x$. 
Let $\xi_1,\xi_2$ be two partitions. We say that $\xi_1$ refines $\xi_2$ denoted $\xi_2\prec \xi_1$ if for $\mu$-a.e. $x\in X$, one has $\xi_1(x)\subset \xi_2(x)$. The \emph{joint} of $\xi_1$ and $\xi_2$ is the partition defined by $\xi_1\vee \xi_2:=\{\xi_1(x)\cap \xi_2(x)\mid x\in X\}$. It is the least fine partition that refines both $\xi_1$ and $\xi_2$. A partition $\xi$ is said to be \emph{measurable} if there is a sequence of finite partitions $\xi_n$ such that
$\xi =\vee_{n=0}^{+\infty}\xi_n. $ 

For a measurable partition $\xi$, Rokhlin showed that there exists a system of \emph{conditional
measures} $(\mu_{x}^\xi)_{x\in X}$ such that:
\begin{itemize}
\item for any $x\in X$, $\mu_x^\xi$ is a probability measure on $\xi(x)$,
\item if $\xi(x)=\xi(y)$, then $\mu_x^\xi=\mu_y^\xi$,
\item for any $\psi \in L^1(X,\mu)$, one has
\begin{equation}
\label{eq:conditional}
\int_X\psi d\mu=\int_X\int_{\xi(x)}\psi(z) d\mu^\xi_x(z)d\mu(x).
\end{equation}
\end{itemize}
Moreover, the system of conditional measures is unique $\mu$-mod $0$, that is, two systems of conditional measures coincide for $\mu$-a.e. $x\in X$.

The main example we will consider is $X=R$ a rectangle and $(\xi(x))_{x\in R}=(\mathcal W^{cs}_R(x))_{x\in R}$ the partition into center stable manifolds. This partition is easily seen to be measurable. For this partition, we will write $\mu^{cs}_x=\mu_x^\xi$ and refer to $(\mu^{cs}_x)_{x\in R}$ as the \emph{conditionals of $\mu$ on the center stable manifolds}. For $q\in R$, define a measure $\tilde \mu_q$ on $\mathcal W^u_R(q)$ by
\begin{equation}
\label{eq:tildemu}
\forall E\subset \mathcal W^u_R(q),\quad \tilde \mu_q(E):=\mu(\cup_{y\in E}\mathcal W^{cs}_R(y)).
\end{equation}
Then \eqref{eq:conditional} can be rewritten for any $q\in R$ as,
\begin{equation}
\label{eq:CPZ}
\forall \phi\in L^1(R,\mu),\quad \mu(\phi)=\int_{\mathcal W^u_R(q)}\int_{\mathcal W^{cs}_R(y)}\phi(z)d\mu^{cs}_y(z)d\tilde \mu^{u}_q(y),
\end{equation}
see \cite[Equation $(2.6)$]{CPZ20}. One defines the \emph{conditionals of $\mu$ on the strong unstable manifolds} $(\mu^{u}_x)_{x\in R}$ by replacing $cs$ by $u$ in the previous definition. Although this is not explicit in the notations, the conditionals  depend on the rectangle $R$. Note however that the ambiguity only consists in a normalizing constant, see \cite[Lemma 2.10]{CPZ20}.
\subsubsection{Metric unstable entropy}
We now recall the definition of the \emph{metric unstable entropy}. For further details, we refer to \cite{HuHuWu}. In this subsection, we consider $f\in \mathcal A_+^{\infty}(\mathbb T^3)$ and $\mu\in \mathcal P_f(\mathbb T^3)$ a  $f$-invariant probability measure.
For a partition $\xi$, define its diameter to be $\mathrm{diam}(\xi):=\sup_{x\in \mathbb T^3}\mathrm{diam}(\xi(x))$. Fix a small $\epsilon_0>0$ and let $\mathcal Q$ be the set of $\mu$-measurable partitions of diameter less than $\epsilon_0$. For $\xi \in \mathcal Q$, we define a finer partition $\eta=:\mathcal Q^u(\xi)$ \emph{subordinated to the unstable foliation} by defining $\eta(x)$ to be the connected component of $\xi(x)\cap \mathcal W^u(x)$ containing $x$, for any $x\in \mathbb T^3$. Let $\mathcal Q^u$ denote the set of all partitions obtained this way.
The \emph{conditional entropy of $\xi$ given $\eta $ with respect to $\mu$} is given by
$$H_\mu(\xi\mid \eta):=-\int_{\mathbb T^3}\ln (\mu^{\eta}_x(\xi(x)))d\mu(x). $$
For $m,n\in \mathbb Z$, define $\xi_{m}^n:=\vee_{i=m}^n f^{-i}\xi$. The \emph{conditional entropy of $f$ with respect to a measurable partition $\xi$ given $\eta$} is 
$$h_\mu(f,\xi\mid \eta)=\limsup_{n\to+\infty}\frac 1n  H_\mu(\xi_0^n\mid \eta).$$
Suppose furthermore that $\mu$ is ergodic, then the \emph{unstable metric entropy of $f$ with respect to $\mu$} is defined to be
\begin{equation}
\label{eq:hmuu}
\forall \xi\in \mathcal Q, \ \forall \eta \in \mathcal Q^u, \quad h^u_\mu(f):=h_\mu(f,\xi\mid \eta),
\end{equation}
where the value does not depend on $\xi$ or $\eta$, see \cite[Theorem A]{HuHuWu}.

\subsubsection{Topological unstable entropy and variational principle.} The topological counter part of the metric unstable entropy is the \emph{topological unstable entropy}. It was introduced by Saghin and Xia \cite{SaXi}. It is equal to the exponential growth rate of the volume of strong unstable balls. For $f\in \mathcal A^\infty_+(\mathbb T^3)$, let
\begin{equation}
\label{eq:hutop1}
 \forall \delta>0, \quad h^u_{\mathrm{top}}(f):=\sup_{x\in \mathbb T^3}\limsup_{n\to+\infty}\frac 1n \ln \mathrm{Vol}\big(f^n(\mathcal W^u(x,\delta))\big),
\end{equation}
where $\mathrm{Vol}$ is the induced volume by the Riemannian metric. The limit does not depend on either $\delta>0$ or the choice of the Riemannian metric. Hu, Hua and Wu \cite[Definition 1.4]{HuHuWu} gave an alternative definition of $h^u_{\mathrm{top}}(f)$ in terms of spanning or separated sets which we will not need in this paper. The link between the two notions of unstable entropies is given by the {variational principle} \eqref{eq:varia11}, see \cite[Theorem D]{HuHuWu}. A measure $\mu \in \mathcal P_f(\mathbb T^3)$ is a \emph{measure of maximal unstable entropy} (or $u$-MME) if $h^u_{\mathrm{top}}(f)=h^u_\mu(f)$. Their set is denoted by $\mathcal M^u (f)$. By the work of Hu, Wu and Zhu \cite[Theorem B]{HuWuZh}, the set $\mathcal M^u(f)$ is always non-empty, convex and compact. Moreover, its extreme points are exactly the ergodic $u$-MMEs.
\subsection{Center Jacobian}
In this subsection, we define the \emph{center Jacobian}. Let
\begin{equation}
\label{eq:Jc}
J^c_f\in C^{\alpha}(\mathbb T^3,\mathbb R),\ \forall x\in \mathbb T^3, \quad J_f^c(x):=-\ln \|df_x|_{E_c(x)}\|,
\end{equation}
where $\|.\|$ denotes the adapted Riemannian metric and $\alpha>0$.
By the thermodynamical formalism (see \cite[Theorem 20.3.7]{HK}), there is a unique \emph{equilibrium state} $\mu_{J^c_f}$ associated to $J_f^c$. It is the unique invariant probability measure such that
\begin{equation}
\label{eq:varia}
h_{\mu_{J^c_f}}(f)+\int_{\mathbb T^3}J^c_f(p)d\mu_{J^c_f}(p)=P(J^c_f):=\sup_{\mu \in \mathcal P_f(\mathbb T^3)}h_{\mu}(f)+\int_{\mathbb T^3}J^c_f(p)d\mu(p),
\end{equation}
where $P(J^c_f)$ is the pressure of the center Jacobian and $h_{\mu}(f)$ is the metric entropy of the map $f$ with respect to $\mu$. We  show an inequality between $h^u_{\mathrm{top}}(f)$ and $P(J^c_f)$.
\begin{lemm}
Let $f\in\mathcal A_+^{\infty}(\mathbb T^3)$. Then one has
\begin{equation}
\label{eq:ineq}
h^u_{\mathrm{top}}(f)\geq P(J^c_f).
\end{equation}
\end{lemm}
\begin{proof}
We use the LeDrappier-Young inequality \cite{LY}, (see also \cite[Corollary A.1]{HuHuWu}) which states that 
\begin{equation}
\label{eq:HuHuWu}
h_{\mu_{J^c_f}}(f)\leq h^u_{\mu_{J^c_f}}(f)+\lambda_c(f,\mu_{J^c_f}), 
\end{equation}
where  $\lambda_c(f,\mu_{J^c_f})$ is the central Lyapunov exponent of $\mu_{J^c_f}$. Recall that it satisfies
\begin{equation*}
\text{for } \mu_{J^c_f}\text{-a.e } x\in \mathbb T^3, \ \forall v\in E_c(x)\setminus \{0\}, \  \quad \lim_{n\to+\infty}\frac 1n \ln \|df^n_x(v)\|=\lambda_c(f, \mu_{J^c_f}).
\end{equation*}
Now, the measure $\mu_{J^c_f}$ is ergodic so by the ergodic theorem and the chain rule, 
$$ \text{for } \mu_{J^c_f}\text{-a.e } x\in \mathbb T^3,\quad \lim_{n\to+\infty}-\frac1n \ln \|df^n_x|_{E_c(x)}\|=\lim_{n\to+\infty}\frac 1n \sum_{k=0}^{n-1}J^c_f(f^kx)=\int_{\mathbb T^3}J^c_f d\mu_{J^c_f}.$$
Plugging this into \eqref{eq:HuHuWu} and using \eqref{eq:varia} yields
\begin{equation}
\label{eq:uvaria}
h^u_{\mu_{J^c_f}}(f)\geq h_{\mu_{J^c_f}}(f)-\lambda_c(f,\mu_{J^c_f})=h_{\mu_{J^c_f}}(f)+\int_{\mathbb T^3}J^c_fd\mu_{J^c_f}=P(J^c_f).
\end{equation}
Using $h^u_{\mu_{J^c_f}}(f)\leq h^u_{\mathrm{top}}(f)$ gives \eqref{eq:ineq}.
\end{proof}

\section{Construction of a Margulis system of measures}
\label{sec3}
Let $f\in\mathcal A_+^{\infty}(\mathbb T^3)$. In this section, we construct a system of $u$-measures for $f$. Recall that there exists a Hölder continuous homeomorphism $h:\mathbb T^3\to \mathbb T^3$ and a hyperbolic matrix $A\in \mathrm{GL}_3(\mathbb Z)$ such that
$ h\circ f=A\circ h, $
see for instance \cite[Theorem 18.6.1]{HK}. The matrix $A$ has three distinct and real eigenvalues $\lambda_s<1<\lambda_c<\lambda_u$ and the corresponding eigenlines are denoted by $E_s^A$, $E_c^A$ and $E_u^A$ respectively. We have $A\in \mathcal A_+^{\infty}(\mathbb T^3)$ and the corresponding foliations are given by $\mathcal W^\bullet_A(x)=x+E_\bullet^A$, for $\bullet=s,c,u$.

Let $H$, $F$ and $A$ denote a lift to $\mathbb R^3$ of $h,f$ and $A$ respectively. For a foliation $\mathcal N$ of $ \mathbb T^3$, we denote by $\widetilde{\mathcal N}$ the lifted foliation of $\mathbb R^3$. Then one has: 
\begin{itemize}
\item the conjugacy lifts to $\mathbb R^3$, that is,
\begin{equation}
\label{eq:conj2}
H\circ F=A\circ H;
\end{equation}
\item  the conjugacy $H$ is isotopic to $\mathrm{Id}$. In particular, there exists $L>0$ such that
\begin{equation}
\label{eq:L}
\|H-\mathrm{Id}\|_{\infty}:=\sup_{x\in \mathbb R^3} d(H(x),x)\leq L,
\end{equation}
where $d$ denotes the Euclidean distance on $\mathbb R^3$;
\item  for any $\bullet=s,c,cs,cu$, one has
\begin{equation}
\label{eq:foliations}
\forall \tilde x\in \mathbb R^3,\quad H\big(\widetilde{ \mathcal W_f^\bullet}(\tilde x)\big)=\widetilde{\mathcal W^\bullet_A}(H(\tilde x));
\end{equation}
\item one has $ H\big(\widetilde{ \mathcal W_f^u}(\tilde x)\big)=\widetilde{\mathcal W^u_A}(H(\tilde x))$ for any $\tilde x\in \mathbb R^3$ if and only if $E_s$ and $E_u$ are jointly integrable.
\end{itemize}
The first three facts were proven by Potrie in \cite{Po} and the last one was proven by Ren, Gan and Zhang in \cite{RGZ}.
Note that the fact that $E_s$ and $E_u$ are \emph{not} jointly integrable is an open and dense condition in $\mathcal A_+^{\infty}(\mathbb T^3)$, see \cite{DW,BRHRHTU,Di}.

We first show the following uniform volume bounds.
\begin{prop}
\label{Mainprop}
Let $\delta>0$. There exists $C_1(\delta),C_2(\delta)>0$ such that for any $\tilde x\in \mathbb R^3$,
\begin{equation}
\label{eq:lowerupper}
\forall n\in \mathbb N,\quad C_1(\delta)\leq \lambda_u^{-n}\mathrm{Vol}(F^n \widetilde{\mathcal W}_f^u(\tilde x,\delta))\leq C_2(\delta).
\end{equation}
\end{prop}
\begin{proof}
By Potrie  \cite[Corollary 7.7]{Po}, the strong unstable foliation is quasi-isometric:
\begin{equation}
\label{eq:eqeq}
\exists C>0, \ \forall \tilde x\in \mathbb R^3, \ \forall \tilde y,\tilde z\in \widetilde W_f^u(\tilde x),\quad d(\tilde y,\tilde z)\leq d^u(\tilde y,\tilde z)\leq Cd(\tilde y,\tilde z)+C,
\end{equation}
where $d$ is the Euclidean distance on $\mathbb R^3$ and $d^u$ denotes the induced distance on $\widetilde W_f^u(\tilde x)$.
Since $F^n \widetilde{\mathcal W}_f^u(\tilde x,\delta)$ is one dimensional and connected, its volume is equal to its $d^u$-length. As a consequence, \eqref{eq:eqeq} gives
$$\mathrm{diam}(F^n \widetilde{\mathcal W}_f^u(\tilde x,\delta))\leq \mathrm{Vol}(F^n \widetilde{\mathcal W}_f^u(\tilde x,\delta))\leq C\mathrm{diam}(F^n \widetilde{\mathcal W}_f^u(\tilde x,\delta))+C, $$
where the diameter is taken with respect to $d$.
Now, using \eqref{eq:L}, we obtain
$$ \mathrm{diam}(H\circ F^n \widetilde{\mathcal W}_f^u(\tilde x,\delta))-L\leq \mathrm{Vol}(F^n \widetilde{\mathcal W}_f^u(\tilde x,\delta))\leq C\mathrm{diam}(H\circ F^n \widetilde{\mathcal W}_f^u(\tilde x,\delta))+C(1+L). $$
Using \eqref{eq:conj2}, this then implies
$$ \mathrm{diam}(A^n\circ H \widetilde{\mathcal W}_f^u(\tilde x,\delta))-L\leq \mathrm{Vol}(F^n \widetilde{\mathcal W}_f^u(\tilde x,\delta))\leq C\mathrm{diam}(A^n\circ H \widetilde{\mathcal W}_f^u(\tilde x,\delta))+C(1+L). $$
Since $\lambda_u$ is simple and the largest eigenvalue of $A$, there is $K>0$ such that 
for any $ \tilde x,\tilde y\in \mathbb R^3,$ one has $\ \|A^n(\tilde x-\tilde y)\|\leq K\lambda_u^n\|\tilde x-\tilde y\|$. In particular, using \eqref{eq:L},
$$C\mathrm{diam}(A^n\circ H \widetilde{\mathcal W}_f^u(\tilde x,\delta))\leq K\lambda_u^n\mathrm{diam}(H \widetilde{\mathcal W}_f^u(\tilde x,\delta))\leq K(2\delta+L)\lambda_u^n.$$
This gives the upper bound in \eqref{eq:lowerupper}.
For $\tilde x\in \mathbb R^3,$ let $\pi_{\tilde x}^u:\mathbb R^3\to E_u^A(\tilde x)$ be the projection parallel to $E_{c}^A(\tilde x)$. Then there exists $c(\delta)>0$ such that for any $ \tilde x\in \mathbb R^3,$
\begin{equation}
\label{eq:lower}
\mathrm{diam}\big(\pi^u_{\tilde x}\big(H \widetilde{\mathcal W}_f^u(\tilde x,\delta)\big)\big)\geq c(\delta).
\end{equation}
Indeed, this follows from the transversality of $E_c$ and $E_u$, the continuity of the bundles and the compactness of $\mathbb T^3.$
In particular, one has
$$\mathrm{diam}(A^n\circ H \widetilde{\mathcal W}_f^u(\tilde x,\delta))\geq \mathrm{diam}\big(\pi_{\tilde x}^u\big(H \widetilde{\mathcal W}_f^u(\tilde x,\delta)\big)\big) \lambda_u^n \geq c(\delta)\lambda_u^n.$$
This gives the lower bound in \eqref{eq:lowerupper}.
\end{proof}
Combining \eqref{eq:hutop1} and \eqref{eq:lowerupper}, and since the volume growth on the lifted foliation is equal to the volume growth on the strong unstable foliation, we obtain:
\begin{corr}
\label{corr12}
For any $f\in\mathcal A_+^{\infty}(\mathbb T^3)$, one has
\begin{equation}
\label{eq:hutop}
e^{h^u_{\mathrm{top}}(f)}=\lambda_u.
\end{equation}
\end{corr}

We now construct the systems of $u$-measures. We start by a definition.
\begin{defi}[System of measures]
\label{def}
For $\bullet=u,cs$, a system of  $\bullet$-measures $\mu^\bullet$ is a family $\mu^\bullet=(\mu^\bullet_x)_{x\in \mathbb T^3}$ such that:
\begin{enumerate}
\item for any $x\in \mathbb T^3$, $\mu^\bullet_x$ is a Borelian measure on $\mathcal W^\bullet(x)$ which is finite on compact subsets of $\mathcal W^\bullet(x)$;
\item for any $x,y\in \mathbb T^3$ such that  $\mathcal W^\bullet(x)= \mathcal W^\bullet(y)$, one has $\mu^\bullet_x=\mu^\bullet_y$.
\end{enumerate}
We say that $\mu^\bullet$ is a measurable $($resp. continuous$)$ system of measures if $\mathbb T^3$ is covered by foliation charts $B$ such that for any $\phi \in C_c^{\infty}(B,\mathbb R)$, the mapping $x\mapsto \mu^\bullet_x(\phi_{|\mathcal W^\bullet_{B}}(x))$ is measurable $($resp. continuous$)$. Here, $\mathcal W^{\bullet}_{B}(x)$ denotes the connected component of $\mathcal W^\bullet(x)\cap B$ containing $x$.
\end{defi}
We start by constructing a measurable system of $u$-measure which has the scaling property \eqref{eq:conf}. The holonomy-invariance \eqref{eq:holhol} will be proved in \S \ref{sec5}.
\begin{prop}[Measurable system of $u$-measures]
\label{leafmeasure}
Let $f\in\mathcal A_+^{\infty}(\mathbb T^3)$. Then there exists a measurable system of  $u$-measure $(\mu^u_x)_{x\in \mathbb T^3}$ such that:
\begin{itemize}
\item for any $x\in \mathbb T^3$, $\mu_x^u$ has full support in $\mathcal W^u(x)$;
\item one has the scaling property
\begin{equation}
\label{eq:Marg}
\forall x\in \mathbb T^3, \quad f^*\mu^u_x=e^{h^u_{\mathrm{top}}(f)}\mu^u_{f^{-1}(x)}.
\end{equation}
\end{itemize}
\end{prop}
\begin{proof}
We use a compactness argument similar to \cite{BFT,CH2} and first construct a system of measure which has full support in each leaf and has the scaling property.
\\
\textbf{There exists a system of $u$-measures with full support in each leaf and the scaling property \eqref{eq:Marg}.}
A test function $\psi$ is the data of $(\psi_x)_{x\in \mathbb T^3}$ where for any $x\in \mathbb T^3$, $\psi_x\in C_c(\mathcal W^u(x))$ is a continuous function with compact support in $\mathcal W^u(x)$. The set of test functions is denoted by $\mathrm{Con}^u(f)$ and it is equipped with the disjoint union topology. Let $\mathrm{Meas}^u(f)$ be the set of systems of $u$-measures.  We endow $\mathrm{Meas}^u(f)$ with the weak topology induced by $\mathrm{Con}^u(f)$. We define $\mathrm{Con}_+^u(f)$ to be the subset of test functions $\psi$ such that for any $x\in \mathbb T^3$, $\psi_x$ is  non-negative and non identically zero. Denote by $\mathrm{Vol}_{\mathcal W^u}\in \mathrm{Meas}^u(f)$ the system of measures induced by the Lebesgue measure on each unstable manifold. Define 
\begin{equation}
\label{eq:chi}
\mathcal X:=\overline{\mathrm{Conv}\big\{\nu^n:=e^{-nh^u_{\mathrm{top}}(f)}(f^*)^n\mathrm{Vol}_{\mathcal W^u}\mid n\in \mathbb N\big\}}\subset \mathrm{Meas}^u(f),
\end{equation}
where $\mathrm{Conv}(X)$ denotes the convex hull of $X$ and $\overline X$ its closure. \begin{lemm}
The subset $\mathcal X\subset \mathrm{Meas}^u(f)$ is compact.
\end{lemm}
\begin{proof}
This follows from Proposition \ref{Mainprop}. Indeed, we first show that given $x_1,x_2\in \mathbb T^3$, $X_1\subset \mathcal W^{u}(x_1)$ and $X_2\subset \mathcal W^{u}(x_2)$ two open and pre-compact sets, there is a constant $\hat e(X_1,X_2)>0$ such that
\begin{equation}
\label{eq:nn}
\forall n\geq 0,\quad \frac{1}{\hat e(X_1,X_2)}\leq \frac{\nu^{n}(X_1)}{\nu^{n}(X_2)}\leq \hat e(X_1,X_2).
\end{equation}
By pre-compactness, there exists a finite set $F$ such that $\overline{X_1}\subset \cup_{y\in F}\mathcal W^u(y,\delta)$. In particular, using \eqref{eq:lowerupper}, one has
$$ \forall n\in \mathbb N, \quad \nu^n(X_1)\leq \sum_{y\in F}\nu^n(\mathcal W^u(y,\delta))\leq |F|\times C_2(\delta).$$
On the other hand, since $X_2$ is open, there exists $z\in X_2$ and $\delta'>0$ such that $\mathcal W^u(z,\delta')\subset X_2$. Using \eqref{eq:lowerupper} again, one has
\begin{equation}
\label{eq:use}\forall n\in \mathbb N, \quad C_1(\delta')\leq \nu^n(\mathcal W^u(z,\delta'))\leq \nu ^n(X_2). 
\end{equation}
This implies that
$$\forall n\in \mathbb N,\quad \frac{\nu^{n}(X_1)}{\nu^{n}(X_2)}\leq \frac{|F|\times C_2(\delta)}{C_1(\delta')}. $$
We deduce \eqref{eq:nn} by exchanging the roles of $X_1$ and $X_2$.

Now, let $\psi\in \mathrm{Con}_+^u(f)$ and $\phi \in \mathrm{Con}^u(f)$ be test functions. There is a $r>0$ such that $A_r:=\psi^{-1}(r,+\infty)$ is open and pre-compact. Choose $A$ open and relatively compact containing the support of $\phi$, then using \eqref{eq:nn} for $A_r$ and $A$, we obtain
$$\forall n\in \mathbb N, \quad \frac{\nu^{n}(\phi)}{\nu^{n}(\psi)} \leq \frac{\|\phi\|_{\infty}\nu^{n}(A)}{r\nu^{n}(A_r)}\leq \frac{\|\phi\|_{\infty}}{r}\hat e(A,A_r)=:\hat e(\phi,\psi).  $$ In total, we have shown that:
\begin{enumerate}
\item for any $\phi\in \mathrm{Con}^u(f)$, there exists  a constant $c(\phi)>0$ such that for any $ \mu \in \mathcal X$, one has $\mu(\phi)\leq c(\phi)$;
\item for any $\phi\in \mathrm{Con}_+^u(f)$, there exists  a constant $c'(\phi)>0$ such that for any $ \mu \in \mathcal X$, one has $\mu(\phi)\geq c'(\phi)$.
\end{enumerate}
By  Tychonoff’s  theorem, this implies that $\mathcal X\subset \mathbb R^{\mathrm{Con}^u(f)}$ is compact for the product topology.
\end{proof}
We now construct the  $u$-system $(\mu^u_x)_{x\in \mathbb T^3}$ as a fixed point. Indeed, define 
\begin{equation}
\label{eq:fixpoint}
S:\mathcal X\to \mathcal X, \quad S(\mu):=e^{-h^u_{\mathrm{top}}(f)}f^*\mu.
\end{equation}
Note that $S$ leaves $\mathcal X$ invariant and is continuous. Since $\mathcal X$ is convex and compact, the Schauder-Tychonoff fixed point Theorem implies that $S$ has a fixed point $\tilde\mu^u$:
\begin{equation}
\label{eq:tildemuu}
\exists \tilde\mu^u\in \mathcal X\quad f^*\tilde\mu^u=e^{h^u_{\mathrm{top}}(f)}\tilde\mu^u 
\end{equation}
For $x\in \mathbb T^3$, let $\varphi \in \mathrm{Con}^u(f)$ such that
\begin{equation}
\label{eq:neweq}
\forall y\in \mathcal W^u(x),\quad \varphi(y)=\begin{cases} 1 \text{ if } y\in \mathcal W^u(\delta/2)
\\0 \text{ if } y\notin \mathcal W^u(\delta).
\end{cases} 
\end{equation}
Using, \eqref{eq:lowerupper} implies that
\begin{equation*}
\forall n\in \mathbb N,\quad C_1(\delta/2)\leq \tilde \nu^n(\varphi)\leq C_2(\delta).
\end{equation*}
By the definition of the weak topology, $\mathcal X\ni \eta \mapsto \eta(\varphi)$ is continuous. Hence $\{ \eta \in \mathcal X \mid \eta(\varphi)\in[ C_1(\delta/2),C_2(\delta)]\}$ is a closed subset of $\mathcal X$ and we deduce that it is equal to $\mathcal X$. In particular,
$$
\forall x\in \mathbb T^3,\quad C_1(\delta/2)\leq  \tilde{\mu}_x^u(\varphi)\leq C_2(\delta).
$$
Since this holds for any $\delta>0$, this gives finally
\begin{equation}
\label{eq:tildemubound}
\forall x\in \mathbb T^3,\quad C_1(\delta/2)\leq \tilde{\mu}_x^u(\mathcal W^u(\delta))\leq C_2(2\delta).
\end{equation}
which shows that $(\tilde\mu^u_x)_{x\in \mathbb T^3}$ has full support in each leaf and is finite on compact sets. Note that it is not clear at this point that $(\tilde\mu^u_x)_{x\in \mathbb T^3}$ is measurable.
\\
\textbf{A definition of the system of measures using dimension theory.} 
We now define the Margulis system $(\mu^u_x)_{x\in \mathbb T^3}$ using the system of measures $(\tilde \mu^u_x)_{x\in \mathbb T^3}$ constructed above and using ideas from dimension theory, following Climenhaga, Pesin et Zelerowicz \cite{CPZ20}. This more explicit construction will allow us to show that the system is measurable.
We recall the definition of a \emph{$u$-Bowen ball}:
\begin{equation}
\label{eq:uBow}
\forall x\in \mathbb T^3, \ \forall n\geq 0, \ \forall \delta>0, \quad \mathcal W^u_n(x,\delta):=\{y\in \mathbb T^3\mid \max_{k=0}^n d^u(f^k x,f^ky)<\delta\}.
\end{equation}
Since we work with an adapted metric, notice that $\mathcal W^u_n(x,\delta)=f^{-n}\mathcal W^u(f^n(y),\delta)$.
For a Borel set $Z \subset \mathcal W^u(x)$, the set of open covers by $u$-Bowen balls is
\begin{equation}
\label{e:covers}
\forall N\geq 0, \quad \mathcal C_N(Z):=\{\mathcal C\subset \mathbb T^3\times \llbracket N,+\infty \llbracket,\ Z\subset \cup_{(y,n)\in \mathcal C}\mathcal W^u_n(y,\delta)\}.
\end{equation}
The system of measures is defined as follows:
\begin{equation}
\label{eq:muu}
\mu^{u,N}_x(Z):=\inf_{\mathcal C\in \mathcal C_N(Z)}\sum_{(y,n)\in \mathcal C}\lambda_u ^{-n} ,\quad \mu^u_x(Z):=\lim_{N\to+\infty} \mu^{u,N}_x(Z).
\end{equation}
\textbf{For any $x\in \mathbb T^3$, $\mu^u_x$ defines a Borelian measure on $\mathcal W^u(x)$.} It is clear from \eqref{eq:muu} that $\mu^u_x(Z)\in[0,+\infty]$ for any Borel subset $Z\subset \mathcal W^u(x)$. We show that $\mu^u_x$ is an \emph{metric outer measure}. Recall that a measure $m$ on a metric space $(X,d)$ is a metric outer measure if for any $A,B\subset X$ such that $d(A,B):=\inf\{d(a,b)\mid a\in A, b\in B\}>0$, one has $m(A\cup B)=m(A)+m(B)$. Every metric outer
measure has the property that all Borel sets are measurable, see for instance \cite[Proposition 12.41]{Roy}.
\\ Suppose that $Z_1,Z_2\subset \mathcal W^u(x)$ are such that $d^u(Z_1,Z_2)>0$. Since the $u$-Bowen balls shrink exponentially fast when $N\to+\infty$, for $N$ large enough, $\mathcal C_N(Z_1\cup Z_2)$ splits as a part that covers $Z_1$ and a part that covers $Z_2$. In particular, one has
$$\forall N\gg 1,\ \mu^{u ,N}_x(Z_1\cup Z_2)=\mu^{u ,N}_x(Z_1)+\mu^{u ,N}_x(Z_2) \ \Rightarrow \ \mu^{u}_x(Z_1\cup Z_2)=\mu^{u}_x(Z_1)+\mu^{u}_x(Z_2).$$
This shows that $\mu^u_x$ is a Borel measure on $\mathcal W^u(x)$.
\\
\textbf{The measure $\mu^u_x$ is finite on compact sets.} Let $K\subset \mathcal W^u(x)$ be a compact subset. Let $N\geq 0$ and consider $\mathcal C \in \mathcal C_N(K)$ such that no proper subcover of $\mathcal C$ is in $\mathcal C_N(K)$. One has 
$$K\subset \bigcup_{(y,n)\in \mathcal C}\mathcal W^u_n(y,\delta)= \bigcup_{(y,n)\in {\mathcal C}}f^{-n}\big(\mathcal W^u(f^ny,\delta)\big).$$
Since $\mathcal W^u(x)$ is one dimensional and since no proper subcover of $\mathcal C$ is in $\mathcal C_N(K)$, for any $z\in K$, there are at most two elements $(y,n)\in {\mathcal C}$ such that $z\in \mathcal W^u_n(y,\delta)$. Hence,
$$\sum_{(y,n)\in \mathcal C}\tilde{\mu}^u_x \big( f^{-n}\big(W^u(f^ny,\delta)\big)\big)\leq 2\tilde{\mu}^u_x(K).$$
Using \eqref{eq:tildemuu} and \eqref{eq:tildemubound}, we obtain
$$\sum_{(y,n)\in \mathcal C}\lambda_u^{-n}\leq  2C_1(\delta/2)^{-1}\tilde{\mu}^u_x(K).$$
This shows that $\mu^u_x(K)\leq 2C_1(\delta/2)^{-1}\tilde{\mu}^u_x(K)$ is finite.
\\
\textbf{The measure $\mu^u_x$ has full support in $\mathcal W^u(x,\delta)$.} Let $x\in \mathbb T^3$ and $\epsilon>0$. We show that $\mu^u_x(\mathcal W^u(x,\epsilon))>0$. Let $\mathcal C\in \mathcal C_N(\mathcal W^u(x,\epsilon))$. Since $\mathcal W^u(x,\epsilon)$ is relatively compact, we can suppose without loss of generality that $ \mathcal C$ is finite. We have
$$\mathcal W^u(x,\epsilon)\subset \bigcup_{(y,n)\in \mathcal C}\mathcal W^u_n(y,\delta)= \bigcup_{(y,n)\in \mathcal C}f^{-n}\big(\mathcal W^u(f^ny,\delta)\big).$$
This implies that 
$$\tilde{\mu}^u_x(\mathcal W^u(x,\epsilon))\leq \sum_{(y,n)\in \mathcal C}\tilde{\mu}^u_x \big( f^{-n}\big(\mathcal W^u(f^ny,\delta)\big)\big).$$
Using \eqref{eq:tildemuu} and \eqref{eq:tildemubound}, we obtain
$$C_1(\epsilon/2)C_2(2\delta)^{-1}\leq  \sum_{(y,n)\in \mathcal C}\lambda_u^{-n}.$$
Passing to the infimum gives $\mu^u_x(\mathcal W^u(x,\epsilon)) \geq C_1(\epsilon)C_2(\delta)^{-1}>0$.
\\
\textbf{The Margulis system has the scaling property \eqref{eq:Marg}.} Let $Z\subset \mathcal W^u(x)$ be a relatively compact Borel set. We need to show that $\mu^u_{f(x)}(f(Z))=\lambda_u\mu^u_x(Z)$. Let $N\geq 0$ be an integer and $\mathcal C\in C_N(Z)$. Then if we define $\mathcal C':=\{(f(x),n-1)\mid (x,n)\in \mathcal C_N(Z)\}$, we see that $\mathcal C'\in \mathcal C_{N-1}(f(Z))$. Moreover, any $\mathcal C'\in \mathcal C_{N-1}(f(Z))$ can be obtained this way and
$$\sum_{(y,m)\in \mathcal C'}\lambda_u^{-m}=\sum_{(z,n)\in \mathcal C}\lambda_u^{-(n-1)}=\lambda_u\sum_{(z,n)\in \mathcal C}\lambda_u^{-n}. $$
Taking the infimum yields $\mu^u_{f(x)}(f(Z))= \lambda_u\mu^u_x(Z)$.
\\
\textbf{The Margulis system $(\mu^u_x)_{x\in \mathbb T^3}$ is measurable.}
Since the limit of a sequence of measurable functions is measurable, \eqref{eq:muu} shows that it suffices to prove that $(\mu^{u,N}_x)_{x\in \mathbb T^3}$ is measurable.
Let $M\in \mathbb N$ with $M>N$. Define
$$\mathcal C_{N,M}(Z):=\{\mathcal C\subset \mathbb T^3\times \llbracket N,M \rrbracket,\ Z\subset \cup_{(y,n)\in \mathcal C}\mathcal W^u_n(y,\delta)\} $$
as well as 
\begin{equation}
\label{eq:muNM}\mu^{u,N,M}_x(Z):=\inf_{\mathcal C\in \mathcal C_{N,M}(Z)}\sum_{(x,n)\in \mathcal C}\lambda_u ^{-n}.  
\end{equation}
Since $\mu^{u,N}_x:=\lim_{M\to+\infty} \mu^{u,N,M}_x$, we only need to show that $(\mu^{u,N,M}_x)_{x\in \mathbb T^3}$ is measurable. To that end, first consider a compact set $K\subset \mathbb T^3$. We prove that  $x\mapsto  \mu^{u,N,M}_x(K_x)$, where $K_x:=K\cap \mathcal W^u_{\mathrm{loc}}(x)$, is upper semi-continuous, thus measurable. By compactness of $K$, $x\mapsto K_x$ is upper semicontinuous. For our purpose, this means that for any $x\in \mathbb T^3$ and any $\epsilon>0$, there is a $\eta>0$ such that if $d(x,y)<\eta$, then $$K_y\subset \mathrm{Hol}^{cs}_{x,y}(U_{\epsilon}(K_x))$$ where $U_{\epsilon}(K_x):=\cup_{z\in K_x}\mathcal W^u(z,\epsilon)$ is the $\epsilon$-neighborhood of $K_x$. Let $\mathcal C\in \mathcal C_{N,M}(K_x)$ with no proper subcover in $\mathcal C_{N,M}(K_x)$. This means that
$$K_x\subset  \bigcup_{(y,n)\in \mathcal C}f^{-n}\big(\mathcal W^u(f^ny,\delta)\big).$$
As before, this implies
$$\sum_{(y,n)\in \mathcal C}\tilde{\mu}^u_x \big( f^{-M}\big(\mathcal W^u(f^ny,\delta)\big)\big)\leq \sum_{(y,n)\in \mathcal C}\tilde{\mu}^u_x \big( f^{-n}\big(\mathcal W^u(f^ny,\delta)\big)\big)\leq 2\tilde{\mu}^u_x(K).$$
Using \eqref{eq:tildemuu} and \eqref{eq:tildemubound}, we obtain
$|\mathcal C|\leq  2C_1(\delta/2)^{-1}\tilde{\mu}^u_x(K_x)\lambda_u^{M}.$
In particular, the infimum in \eqref{eq:muNM} can be taken over partitions of uniformly bounded cardinality. As a consequence, the function in \eqref{eq:muNM} takes its values in a finite set and thus attains its minimum. We now suppose that $\mathcal C\in \mathcal C_{N,M}(K_x)$ is such that 
$$\mu^{u,N,M}_x(K_x)=\sum_{(y,n)\in \mathcal C} \lambda_u^{-n}.$$
Since $K_x$ is closed, there exists $\epsilon>0$ and $\delta'<\delta$ such that 
$$U_\epsilon(K_x)\subset   \bigcup_{(y,n)\in \mathcal C}f^{-n}\big(W^u(f^ny,\delta')\big).$$
Let $y\in \mathbb T^3$ be $\eta$-close to $x$ where $\eta$ is defined by the upper semi-continuity of $K$. Up to taking $\eta$ even smaller and since all the foliations are continuous, we have 
$$\mathrm{Hol}^{cs}_{f^nx,f^ny}\big(\mathcal W^u(f^nz,\delta')\big)\subset \mathcal W^u( \mathrm{Hol}^{cs}_{f^nx,f^ny}\circ f^n z, \delta)=\mathcal W^u(f^n \circ \mathrm{Hol}^{cs}_{x,y}(z),\delta). $$
In other words, we have shown that $\mathcal C':=\{( \mathrm{Hol}^{cs}_{x,y}(z),n)\mid (z,n)\in \mathcal C\} $ is an element of $\mathcal C_{N,M}(K_y)$. Moreover, one has
$$\mu^{u,N,M}_x(K_x)= \sum_{(z,n)\in \mathcal C}\lambda_u^{-n} = \sum_{(w,n)\in \mathcal C'}\lambda_u^{-n}. $$
Passing to the infimum, we thus have $\mu^{u,N,M}_y(K_y)\leq \mu_x^{u,N,M}(K_x)$ which finishes the proof for compact sets $K$. 

Now, let $\mathcal K$ be the collection of all subsets $Z\subset \mathbb T^3$ such that $x\mapsto \mu^{u,N,M}_x(Z_x)$ is measurable. Note that this is a Dynkin system, i.e., 
\begin{itemize}
\item If $Z_1,Z_2\in \mathcal K$ and $Z_1\subset Z_2$, then $Z_2\setminus Z_1\in \mathcal K$.
\item If $Z_1\subset Z_2\subset \ldots\subset Z_n\subset \ldots \in \mathcal K$, then $\cup_{n\in \mathbb N}Z_n\in \mathcal K$.
\item $\mathbb T^3\in \mathcal K$.
\end{itemize}
Since $\mathcal K$ contains the closed subsets, Dynkin $\pi$-$\lambda$ theorem implies that it must contain the $\sigma$-algebra generated by the closed subsets. In other words, $\mathcal K$ contains all Borel sets which concludes the proof.
\end{proof}
Having constructed a measurable system of $u$-measures, we now characterize $u$-MMEs by their conditionals along the strong unstable foliation. We recall the following result due to Buzzi, Fisher and Tazhibi \cite{BFT}, see also \cite[Proposition 5.3]{Ta}.
\begin{prop}
Let $f\in \mathrm{Diff}^{\infty}(M)$ be a partially hyperbolic diffeomorphism on a closed manifold $M$ with a measurable Margulis system of measures $(\nu^u_{x})_{x\in M}$ such that 
\begin{itemize}
\item for any $x\in M$, $\nu^u_{x}$ is fully supported in $\mathcal W^u(x)$,
\item there is $D_u>0$ such that for any $x\in M$, 
$ f_*\nu^u_x=D_u\nu^u_{f(x)}.$
\end{itemize}
Then for any invariant measure $\nu$, one has
$ h^u_{\nu}(f)\leq \ln (D_u),$
with equality if and only if the disintegration of $\nu$ along $\mathcal W^u$ is given by $\nu^u_x$ $ ($up to a constant rescaling$)$, $\nu$-a.e.
\end{prop}
Using Corollary \ref{corr12} and Proposition \ref{leafmeasure}, we deduce Corollary \ref{cor}.
\section{The first Ruelle resonance on $2$-forms.}
\label{sec4}
In this section, we study the first resonance of the Anosov diffeomorphism $f$ when acting on the bundle of differential $2$-forms. Recall that the space of resonant states $\mathrm{Res}_k(f)$ is defined in \eqref{eq:res} and that $\lambda\in \mathrm{Res}_k(f)\cap \mathbb R_+$ is the first resonance for the action on $k$-forms if $\mathrm{Res}_k(f)\subset \{z\in \mathbb C,|z|\leq \lambda\}.$ We will show the following result.
\begin{prop}[First Ruelle resonance]
\label{prop1st}
Let $f\in\mathcal A_+^{\infty}(\mathbb T^3)$. Then one has
\begin{equation}
\label{eq:equality}
\sup\{|z|, z\in \mathrm{Res}_2(f)\}=\lambda_u=e^{h_{\mathrm{top}}^u(f)}=e^{P(J^c_f)}.
\end{equation}
Moreover, $\lambda_u$ is the {first Pollicott-Ruelle resonance}, it has no Jordan block and has multiplicity equal to one. Finally, the trace of the spectral projector of $f$ at $z=\lambda_u$ is equal to $\mu_{J^c_f}$.
\end{prop}
We note that as a direct consequence of  \eqref{eq:equality}  and \eqref{eq:uvaria}, we obtain.
\begin{corr}
\label{lemmmuJ}
Let $f\in\mathcal A_+^{\infty}(\mathbb T^3)$.
The equilibrium state $\mu_{J^c_f}$ is a $u$-MME.
\end{corr}
\subsection{Ruelle resonances}
\label{secRuelle}
In this subsection, we recall some facts about the functional approach for Anosov diffeomorphisms. For a detailed introduction to the subject, we refer to \cite{Ba}. For $0\leq k \leq 3$, define the vector bundle of smooth $k$-forms
$\Omega_k= C^{\infty}(\mathbb T^3; \Lambda^k T^* (\mathbb T^3)). $
The diffeomorphism $f$ acts on $\Omega_k$ by pushforward:
\begin{equation}
\label{eq:Lie}
f_{*,k}: \Omega_k\to \Omega_k, \quad f _{*,k} \omega:=f_*\omega.
\end{equation}
For an Anosov diffeomorphism, one can associate to $f_{*,k}$ a discrete spectrum $\mathrm{Res}_k(f)$, the \emph{Pollicott-Ruelle resonances}, by making it act on specially designed \emph{anisotropic spaces}, see for instance \cite{BKL,BT,GouLiv,Fau08}. 

For any $0\leq k \leq 3$, the resolvent 
$R_k(\lambda):=(f_{*,k} -\lambda)^{-1}:L^{2}(\mathbb T^3;\Lambda^k T^* (\mathbb T^3))\to L^{2}(\mathbb T^3;\Lambda^k T^* (\mathbb T^3)), $
defined for $\lambda\in \mathbb C$ such that $|\lambda|\gg1$ admits a meromorphic extension $R_k(\lambda)$ to $\mathbb C\setminus \{0\}$
\begin{equation}
\label{eq:resol}
R_k(\lambda): \Omega_k\to \mathcal D'_k(\mathbb T^3),
\end{equation}
where $\mathcal D'_k(\mathbb T^3)$ denotes the space of distributional currents of degree $k$, i.e., the dual space of $\Omega_{3-k}$, see \cite[Chapter 10.1]{Lef}. We will write $(\cdot,\cdot)_{\Omega_k \times \mathcal D'_{3-k}}$ for the pairing between $\Omega_k$ and $\mathcal D'_{3-k}$. More precisely,  by the work of Faure, Roy and Sj{\"o}strand \cite{Fau08}, there exists a family of Hilbert spaces $(\mathcal H^s_+ (\mathbb T^3;\Lambda^k T^* (\mathbb T^3)))_{s>0}$ such that:
\begin{itemize}
\item the space $\Omega_k$ is densely included in $\mathcal H^s_+ (\mathbb T^3;\Lambda^k T^* (\mathbb T^3))$ (see  \cite[Lemma 9.1.13]{Lef}) and one has $H^s\subset \mathcal H^s_+ \subset H^{-s} $, where $H^s $ is the usual $L^2$-Sobolev space of order $s$ (see  \cite[Lemma 9.1.14]{Lef});
\item
 there exists $c>0$ such that for any $s>0$, the resolvent  
$
R_k(\lambda)= \mathcal H^s_+\to \mathcal H^s_+$
is well defined, bounded and holomorphic for $\{|\lambda|\gg 1\}$, and has a meromorphic extension to $\{|\lambda|> e^{-cs}\}$ independent of any choice made in the construction. 
\end{itemize}
Applying the result to $f^{-1}$, we obtain a family of Hilbert spaces $(\mathcal H^s_- (\mathbb T^3, \Lambda^k T^* (\mathbb T^3)))_{s>0}$, dual to $(\mathcal H^s_+ (\mathbb T^3, \Lambda^k T^* (\mathbb T^3)))_{s>0}$ on which the resolvent of the pullback of $f$ has a meromorphic extension, see \cite{Fau08}. We will write $(\cdot,\cdot)_{\mathcal H^s_+ \times \mathcal H^s_-}$ for the pairing between $\mathcal H^s_+$ and $\mathcal H^s_-$.

The poles of the extension are intrinsic and are called the \emph{Pollicott-Ruelle resonances $\mathrm{Res}_k(f)$ of $f_{*,k}$}. For any $\lambda_0\in \mathrm{Res}_k(f)$, the spectral projector at $\lambda_0$ is given by
$$ \Pi_k^{\lambda_0}=-\frac{1}{2i\pi}\int_{\gamma}R_k(z)dz,$$
where $\gamma$ is a small loop around $\lambda_0$ and has finite rank. The (algebraic) multiplicity of $\lambda_0$ is given by the rank of the spectral projector:
\begin{equation}
\label{eq:mult}
\forall \lambda_0\in \mathrm{Res}_k(f), \quad m_k(\lambda_0):=\mathrm{rank}(\Pi_k^{\lambda_0})\in \mathbb N.
\end{equation}
Near a resonance $\lambda_0\in \mathrm{Res}_k(f)$, the resolvent admits a Laurent expansion of the form
$$R_k(\lambda)=R_k^H(\lambda)-\sum_{j=1}^{N(\lambda_0)}\frac{(f_{*,k} -\lambda_0)^{j-1}\Pi^{\lambda_0}_k}{(\lambda-\lambda_0)^j}, $$
where $R_k^H(\lambda)$ is the holomorphic part near $\lambda_0$. The \emph{generalized resonant states} are the elements in the range of the spectral projector.
$$
\mathrm{Res}_{k,\lambda_0,\infty}(f):=\Pi^{\lambda_0}_k(\mathcal H^s_+)=\Pi^{\lambda_0}_k(\Omega_k)=\{u\in \mathcal H^s_+\mid (f_{*,k}-\lambda_0)^{N(\lambda_0)}u=0\}.
$$
Since the meromorphic extension is independent of the particular choice of anisotropic space $\mathcal H^s_+$, the space of generalized resonant states does not depend on $\mathcal H^s_+$. One way to see this is to use the following equivalent characterization of generalized resonant states. Define a dual decomposition to the Anosov decomposition:
\begin{equation}
\label{eq:dual}
T^*(\mathbb T^3)=E_s^*\oplus E_{cu}^*=E_{cs}^*\oplus E_u^*, \quad E_\bullet^*(E_\bullet)=0, \ \bullet=s,cu,cs,u.
\end{equation}
Then one has, see \cite[Lemma 5.12]{DFG},
$$
\mathrm{Res}_{k,\lambda_0,\infty}(f)=\{u\in \mathcal D'_k(\mathbb T^3)\mid (f_{*,k}-\lambda_0)^{N(\lambda_0)}u=0,\ \mathrm{WF}(u)\subset E_{cu}^*\},
$$
where $\mathrm{WF}(u)$ denotes the wavefront set of a current $u$, see \cite[Chapter VIII]{Hor}.
The \emph{resonant states} are defined as 
\begin{equation}
\label{eq:resonant}
\mathrm{Res}_{k,\lambda_0}(f):=\{u\in \mathcal D'_k(\mathbb T^3)\mid (f_{*,k}-\lambda_0)u=0,\ \mathrm{WF}(u)\subset E_{cu}^*\}. \end{equation}
We will say that $f$ has no Jordan block at $\lambda_0$ if $N(\lambda_0)=1$, that is, if $\mathrm{Res}_{k,\lambda_0}(f)=\mathrm{Res}_{k,\lambda_0,\infty}(f)$. Note that applying the previous results to $f^{-1}$ yields notions of (generalized) co-resonant states, see \eqref{eq:cores}.
\subsection{Zeta function}
We recall the definition of the \emph{Ruelle zeta function}. In our setting, this function is also known as the \emph{Artin-Mazur zeta function}.  For $n\geq 1$, let 
\begin{equation}
\label{eq:Fixf^n}
\mathcal F_n(f):=\{x\in \mathbb T^3\mid f^n(x)=x\}.
\end{equation}
The Ruelle zeta function is given by:
\begin{equation}
\label{eq:Zeta}
\forall z\in \mathbb C,\  |z|\ll 1, \quad \zeta_{f}(z):=\exp \left(\sum_{m=1}^{+\infty}\frac{z^m}{m}\mathrm{Card}(\mathcal F_m(f))\right).
\end{equation}
The zeta function $\zeta_f$ has a meromorphic extension (actually rational extension in our case) to the complex plane. This is a special case of the works of Liverani-Tsujii \cite{LiTs} and Baladi-Tsujii \cite{BaTs}. In fact, for an Anosov diffeomorphism on a torus, one can compute explicitly $\zeta_f$ in terms of its action on the homology $A$. The following Lemma is due to Manning \cite{Man}.
\begin{lemm}
Let $f\in\mathcal A_+^{\infty}(\mathbb T^3)$ and let $A\in \mathrm{GL}_3(\mathbb Z)$ be its action on homology. Then
\begin{equation}
\label{eq:zetafA}
\forall z\in \mathbb C, \quad \zeta_f(z)=\zeta_A(z)=\frac{(1-z\lambda_s\lambda_c)(1-z\lambda_s\lambda_u)(1-z\lambda_c\lambda_u)}{(1-z\lambda_s)(1-z\lambda_c)(1-z\lambda_u)}.
\end{equation}
\end{lemm}
\begin{proof}
Since $f$ preserves orientation, the Lefschetz fixed point Theorem gives
$$ \mathrm{Card}\big(\mathcal F_m(f)\big)=\sum_{k=0}^3(-1)^k \tr(\tilde f_{k,*}^m),$$
where $\tilde f_{k,*}:H_k(\mathbb T^3, \mathbb R)\cong \mathbb Z^{\binom{3}{k}}\to H_k(\mathbb T^3, \mathbb R)$ denotes the action of $f$ on the $k$-th homology group of $\mathbb T^3$. Since $\tilde f_{k,*}=\tilde A_{k,*}$, an easy computation using \eqref{eq:Zeta} yields
$$ \zeta_f(z)=\zeta_A(z)=\prod_{k=0}^3 \mathrm{det}(\mathrm{Id}-z\tilde A_{k,*})^{(-1)^k}=\frac{(1-z)(1-z\lambda_s\lambda_c)(1-z\lambda_s\lambda_u)(1-z\lambda_c\lambda_u)}{(1-z)(1-z\lambda_s)(1-z\lambda_c)(1-z\lambda_u)},$$
which simplifies into \eqref{eq:zetafA}.
\end{proof}
\subsection{Dynamical determinants}
In this subsection, we recall the link between the spectral theory of the diffeomorphism $f$ on the anisotropic spaces and the periodic orbits of $f$.
 We define the \emph{dynamical determinants}. For $k\in \llbracket 0, 3\rrbracket$, let
\begin{equation}
\label{eq:D_f,k}
\forall z\in \mathbb C, |z|\ll 1, \quad \mathcal D_{f,k}(z):=\exp\left( -\sum_{m=1}^{+\infty} \frac {z^m}m\sum_{x\in \mathcal F_m(f)}\frac{\tr(\Lambda^kdf^{-m}(x))}{|\mathrm{det}(\mathrm{Id}-df^{-m}(x))|}\right).
\end{equation}
We will need the following result, see for instance \cite[Theorem 6.2]{Ba}. For any $k\in \llbracket 0, 3\rrbracket$, the dynamical determinant $\mathcal D_{f,k}$ admits a \emph{holomorphic} extension to $\mathbb C$. Moreover,
\begin{equation}
\label{eq:zero}
\forall z\in \mathbb C^*, \quad \mathcal D_{f,k}(z)=0 \ \iff \ z^{-1}\in \mathrm{Res}_k(f),
\end{equation}
 that is, $z$ is a zero of $ \mathcal D_{f,k}$ if and only if  $z^{-1}$ is a Pollicott-Ruelle resonance of $f$ for the action on $k$-forms. Moreover, the multiplicity of the zero coincides with the algebraic multiplicity of $z^{-1}$. Equation \eqref{eq:zero} gives a powerful link between the spectral theory of $f$ and weighted sums over periodic points. We will show that $\mathcal D_{f,2}(z)$ is essentially given (on the critical circle of convergence) by the exponentials of Birkhoff sums of $J^c_f$ over periodic points. Using Bowen's  equidistribution results, we obtain the following lemma. Recall that $E_c$ is  Hölder, so that $J^c_f$ has a unique equilibrium state $\mu_{J^c_f}$.
 \begin{lemm}
 \label{lemm1}
 Let $f\in\mathcal A_+^{\infty}(\mathbb T^3)$, then one has
 $$\sup\{|z|, z\in \mathrm{Res}_2(f)\}=e^{P(J^c_f)}, \quad e^{P(J^c_f)}\in \mathrm{Res}_2(f). $$
 We will say that $e^{P(J^c_f)}$ is the first resonance of $f$ for the action on $2$-forms. Moreover, $e^{P(J^c_f)}$ has no Jordan block, has multiplicity $1$ as a Pollicott-Ruelle resonance and it is the only resonance on the \emph{critical circle} $\{ z\in \mathbb C, |z|=e^{P(J^c_f)}\}$.
 \end{lemm}
 \begin{proof}
 From \eqref{eq:zero}, it suffices to locate the first zero of $\mathcal D_{f,2}$ and compute its multiplicity. Let $x\in \mathbb T^3$ and let $n\geq 1$. Since $f\in\mathcal A_+^{\infty}(\mathbb T^3)$, the linear map $df^{-m}(x)$ has three eigenvalues
 $\rho^m_\bullet(x):=\mathrm{det}(df^{-m}(x)|_{E_\bullet(x)})$ for $ \bullet =s,c,u. $ Moreover, there exists constants $C,\nu>0$ such that for any  $x\in \mathbb T^3$ and $m\geq 1$,
 \begin{equation}
 \label{eq:eigen fn}
 \rho^m_s(x)\geq Ce^{\nu m}\geq 1 \geq C e^{-\nu m}\geq \rho^m_c(x)\geq C e^{-\nu m} \rho^m_c(x)\geq \rho^m_u(x).
 \end{equation}
 In particular, we deduce that, when $m\to+\infty$, 
 \begin{equation}
 \label{eq:equiv}
 \begin{split}
 &|\mathrm{det}(\mathrm{Id}-df^{-m}(x))|=\rho^m_s(x)\big(1+O(e^{-\nu m})\big),
 \\ &\tr(\Lambda^2df^{-m}(x))=\rho^m_s(x)\rho^m_c(x)\big(1+O(e^{-\nu m})\big),
 \end{split}
 \end{equation}
  uniformly in $x\in \mathbb T^3$. Plugging \eqref{eq:equiv} into \eqref{eq:D_f,k} gives
  $$ \sum_{x\in \mathcal F_m(f)}\frac{\tr(\Lambda^2df^{-m}(x))}{|\mathrm{det}(\mathrm{Id}-df^{-m}(x))|}=\sum_{x\in \mathcal F_m(f)}\rho_c^m(x)\big(1+O(e^{-\nu m})\big)\sim_{m\to+\infty}\sum_{x\in \mathcal F_m(f)}\rho_c^m(x).$$
Now, using the chain rule and \eqref{eq:Jc}, we obtain
$$ \forall m \geq 1,\ \forall x\in \mathcal F_m(f), \quad \rho^m_c(x)=\mathrm{det}(df^m(f^{-m}x)|_{E_c(x)})^{-1}=\exp\left( \sum_{k=0}^{m-1}J_f^c(f^kx) \right).$$
Since $\rho^m_c(x)>0$, the zeros of $\mathcal D_{f,2}$ on the critical circle are the same as the zeros of
$$\zeta_{f, J^c_f}(z):=\exp\left(-\sum_{m=1}^{+\infty} \frac{z^m}m\sum_{x\in \mathcal F_m(f)} \exp\left( \sum_{k=0}^{m-1}J_f^c(f^kx) \right)\right).$$
We recognize the Ruelle zeta function with potential $J^c_f$. Since $f\in\mathcal A_+^{\infty}(\mathbb T^3)$ is conjugate to a linear automorphism of $\mathbb T^3$, it is topologically mixing.
 By \cite[Theorem 1]{Ru}, there is a $R=R(J^c_f)>e^{-P(J^c_f)}$ such that $\zeta_{f, J^c_f}$ has no zero in $ \overline{B}(0,R)$ except for a simple zero at  $e^{-P(J^c_f)}$. This  shows that $e^{P(J^c_f)}$ is a resonance of algebraic multiplicity equal to $1$ and that  $e^{P(J^c_f)}$ is the only resonance on the critical circle.
 \end{proof}
 The dynamical determinants can be linked to the zeta function, 
\begin{equation}
\label{eq:zetatrace}
\zeta_f(z)=\prod_{k=0}^3 \mathcal D_{f, k}(z)^{(-1)^{k+1}}=\frac{\mathcal D_{f, 1}(z)\mathcal D_{f, 3}(z)}{\mathcal D_{f, 0}(z)\mathcal D_{f, 2}(z)},
\end{equation}
see \cite[Equation $(6.38)$]{Ba}. Using \eqref{eq:zetafA}, we also show the following.
\begin{lemm}
\label{lemmfirst}
Let $f\in\mathcal A_+^{\infty}(\mathbb T^3)$. Then 
$$\sup\{|z|, z\in \mathrm{Res}_2(f)\}=\lambda_u. $$
\end{lemm}
\begin{proof}
Combining \eqref{eq:zetatrace} and \eqref{eq:zetafA}, we get
\begin{equation}
\label{eq:useful}\frac{\mathcal D_{f, 1}(z)\mathcal D_{f, 3}(z)}{\mathcal D_{f, 0}(z)\mathcal D_{f, 2}(z)}= \frac{(1-z\lambda_s\lambda_c)(1-z\lambda_s\lambda_u)(1-z\lambda_c\lambda_u)}{(1-z\lambda_s)(1-z\lambda_c)(1-z\lambda_u)}.
\end{equation}
Since the dynamical determinants are holomorphic, the pole $\lambda_u^{-1}$ can only come from either a zero of $\mathcal D_{f, 0}$ or a zero of $\mathcal D_{f, 2}$. In other words, we have either $\lambda_u\in \mathrm{Res}_0(f)$ or  $\lambda_u\in\mathrm{Res}_{2}(f)$. Now, for any $g\in C^{\infty}(\mathbb T^3)$, one has $\|f_*g\|_{\infty}\leq \|g\|_{\infty}$. Hence, for any $z\in \mathbb C$ such that $|z|>1$, the resolvent $(f_*-z)^{-1}:L^{\infty}(\mathbb T^3)\to L^{\infty}(\mathbb T^3)$ is defined by 
$$ (f_*-z)^{-1}g=z^{-1}\sum_{m=0}^{+\infty}\frac{f^n_* g}{z^n}.$$
In particular, this shows that $\mathrm{Res}_0(f)\subset \{ z \in \mathbb C, \ |z|\leq 1\}$ and thus $\lambda_u\notin \mathrm{Res}_0(f)$. We deduce that $\lambda_u\in \mathrm{Res}_2(f)$ is a Pollicott-Ruelle resonance. Using Lemma \ref{lemm1}, this shows that $\lambda_u\leq e^{P(J^c_f)}$. By Corollary \ref{corr12}, this means that $h^u_{\mathrm{top}}(f)\leq P(J^c_f)$. Using \eqref{eq:ineq} we deduce $h^u_{\mathrm{top}}(f)= P(J^c_f)$, which concludes the proof of the lemma and of \eqref{eq:equality}.
\end{proof}
We show that the spectral projector at the first resonance is equal to~$\mu_{J^c_f}$. \begin{lemm}
\label{lemm2}
Let $f\in\mathcal A_+^{\infty}(\mathbb T^3)$. The trace of the spectral projector at the first resonance $e^{P(J^c_f)}$ is given by the product of a resonant state $\theta$ and a co-resonant state $\nu$ and is equal to the equilibrium state $\mu_{J^c_f}$.
\end{lemm}
\begin{proof} Let $g\in C^{\infty}(\mathbb T^3,\mathbb R_+)$ be a smooth and positive function. 
 The Guillemin trace formula (see for instance \cite[Proposition 6.3]{BaTs}) gives for any $m\geq 0$, 
$$\tr^{\flat}\big(gf_{*,2}^m\big)=\sum_{x\in \mathcal F_m(f)}\frac{g(x)\tr(\Lambda^2df^{-m}(x))}{|\mathrm{det}(\mathrm{Id}-df^{-m}(x))|}, $$
where $\tr^{\flat}$ denotes the \emph{flat trace} and $f_{*,2}$ is defined in \eqref{eq:Lie}. For any $z\in \mathbb C$, such that $|z|\ll 1$, define $Z_{g,f}(z):=\tr^{\flat}\big(g(z-f_{*,2})^{-1}\big)$. For $|z|\ll 1$,
\begin{equation}
\label{eq:Zfg}
Z_{g,f}(z)=\sum_{m=0}^{+\infty}z^{-m-1}\tr^{\flat}\big(gf_{*,2}^m\big)=\sum_{m=0}^{+\infty}z^{-m-1}\sum_{x\in \mathcal F_m(f)}\frac{g(x)\tr(\Lambda^2df^{-m}(x))}{|\mathrm{det}(\mathrm{Id}-df^{-m}(x))|}.
\end{equation}
Note that $Z_{g,f}$ admits a meromorphic extension to $\mathbb C\setminus \{0\}$ since the resolvent does.
Using Lemma \ref{lemm1}, we see that $Z_{g,f}$ has a first simple pole at $z=\lambda_u$ and no other pole outside of $B(0, \lambda_u^{-1}-\epsilon)$ for some $\epsilon>0$. We use Lemma \ref{lemm1} and write the spectral projector at the first resonance $\lambda_u$ as 
$ \Pi_0(\lambda_u)=\nu(\cdot) \theta,$
where $\theta$ (resp.  $\nu$) is a resonant (resp. co-resonant state) which are normalized so that $(\nu,\theta)_{\mathcal H^s_-\times \mathcal H^s_+}=1$. Recall that near $z=\lambda_u^{-1}$, 
\begin{align*}\tr^{\flat}\big(g(z-f_{*,2})^{-1}\big)&=-\tr^{\flat}\big(g(f_{*,2}-z)^{-1}\big)=\tr^{\flat}\Big(g\big(\tfrac{\Pi_0}{z-\lambda_u}+h(z)\big) \Big),
\end{align*}
for some holomorphic map $h$.
In particular, the residue at $\lambda_u$ is given by
\begin{align*}\mathrm{Res}_{\lambda_u}(Z_{g,f})&=\tr(g\Pi_0(\lambda_u))=\nu(g\times  \theta)=(\nu, g\theta)_{\mathcal H^s_-\times \mathcal H^s_+}=:(\nu\wedge \theta)(g).
\end{align*}
Next, we use Cauchy's formula on \eqref{eq:Zfg} to obtain, for some $\rho<\lambda_u^{-1}$,
\begin{align*}\sum_{x\in \mathcal F_m(f)}\frac{g( x)\tr(\Lambda^2df^{-m}(x))}{|\mathrm{det}(\mathrm{Id}-df^{-m}(x))|}=&\mathrm{Res}_{\lambda_u}(Z_{g,f}(z)z^{m+1})+\frac 1{2i\pi}\int_{\rho \mathbb S^1}Z_{g,f}(z)z^{m-1}dz.
\end{align*}
Since $Z_{g,f}$ has a simple first pole $\lambda_u$, one has
$$\sum_{x\in \mathcal F_m(f)}\frac{g(x)\tr(\Lambda^2df^{-m}(x))}{|\mathrm{det}(\mathrm{Id}-df^{-m}(x))|} =(\nu\wedge \theta)(g)\lambda_u^{m}(1+O(e^{-\epsilon m})),$$
for some $\epsilon>0$. Using \eqref{eq:equiv}, this gives
\begin{equation}
\label{eq:Bowen}
(\nu\wedge \theta)(g)=\lim_{m\to+\infty} \lambda_u^{-m}\sum_{x\in \mathcal F_m(f)}\exp\left( -\sum_{k=0}^{m-1}J_f^c(f^kx) \right)g(x).
\end{equation}
Using Lemmas \ref{lemm1} and \ref{lemm2}, for $g\equiv 1$, we obtain
\begin{align*}\lim_{m\to+\infty} &\lambda_u^{m}\sum_{x\in \mathcal F_m(f)}\mathrm{det}(df^m(x)|_{E_c(x)})=\lim_{m\to+\infty} e^{mP(J^c_f)}\sum_{x\in \mathcal F_m(f)}\mathrm{det}(df^m(x)|_{E_c(x)})=1. 
\end{align*}
This means that \eqref{eq:Bowen} is Bowen's formula for the equilibrium state (see for instance $\mu_{J^c_f}$ \cite[Theorem 20.3.7]{HK}):
$$(\nu\wedge \theta)(g)=\lim_{m\to+\infty} \frac{\sum_{x\in \mathcal F_m(f)}\exp\left( -\sum_{k=0}^{m-1}J_f^c(f^kx) \right)g(x)}{\sum_{x\in \mathcal F_m(f)}\exp\left( -\sum_{k=0}^{m-1}J_f^c(f^kx) \right)}=\mu_{J^c_f}(g), $$ which concludes the proof.
\end{proof}
 Recall that by Proposition \ref{prop1st}, we know that the space of (co)-resonant states at $\lambda_u$ is one dimensional and thus spanned by a (co)-resonant state $\theta$ and $\nu$.  \begin{lemm}
\label{lemm0}
For any $s>0$ large enough and for any  $\omega \in \mathcal H^s_+$, one has
\begin{equation}
\label{eq:conv}
\exists c_\omega\in \mathbb R, \quad \lim_{n\to+\infty}\lambda_u^{-n}(f_*)^n\omega= c_\omega \theta,
\end{equation}
where the limit exists in the anisotropic space $\mathcal H^s_+$. Similarly, for $\eta \in \mathcal H^s_-$, one has
\begin{equation}
\label{eq:conv2}
\exists c_\eta\in \mathbb R, \quad \lim_{n\to+\infty}\lambda_u^{-n}(f^*)^n\eta= c_\eta \nu,
\end{equation}
where the limit exists in the anisotropic space $\mathcal H^s_-$.
\end{lemm}
\begin{proof}
 By Lemma \ref{lemm1}, there is a $R>0$ such that $\lambda_u$ is the only Pollicott-Ruelle resonance in $\mathbb C\setminus B(0,R)$ and it is simple with no Jordan block. We fix $s$ such that $e^{-cs}<R$, so that $
R_k(\lambda)= \mathcal H^s_+\to \mathcal H^s_+$ has a meromorphic extension to $\{|\lambda|> R\}$. In particular, denoting by $\Pi_0:\mathcal H^s_+\to \mathcal H^s_+$ the spectral projector on $\mathrm{Span}(\theta)$, there exists a bounded operator $K :\mathcal H^s_+\to \mathcal H^s_+$ satisfying
$$ f_{*,k}=\lambda_u \Pi_0+K, \quad K\Pi_0=\Pi_0K=0, \quad r(K)<\lambda_u,$$
where $r(K)$ is the spectral radius of $K$. For any $\epsilon>0$, recall that one has
$$\exists C>0, \quad\forall n\geq0, \quad \|K^n\|_{\mathcal H^s_+\to \mathcal H^s_+}\leq C(r(K)+\epsilon)^n. $$
If $\epsilon>0$ is chosen so that $r(K)+\epsilon<\lambda_u$, then, as bounded operators $\mathcal H^s_+\to \mathcal H^s_+$, 
$$ \lambda_u^{-n}f_{*,k}^n=\lambda_u^{-n}(\lambda_ u\Pi_0+K)^n\to_{n\to+\infty}\Pi_0. $$
Let $\omega \in \mathcal H^s_+$, then there is $c_\omega\in \mathbb R$ such that $\Pi_0(\omega)=c_\omega \theta$. We obtain \eqref{eq:conv} by applying the previous convergence to~$\omega$. The proof of \eqref{eq:conv2} is analogous.
\begin{rem}
\label{rem:Gabriel}
The resonant $($resp. co-resonant$)$ state $\theta$ $($resp. $\nu)$ is closed. Indeed, since $d$ commutes with the pushforward (resp. pullback) and preserves the spaces of currents with wavefront set in $E_s^*$ $($resp. $E_u^*)$, one has $d\theta \in \mathrm{Res}_{3,\lambda_u}(f)$ $($resp. $d\nu \in \mathrm{Res}^*_{2,\lambda_u}(f))$. A similar argument as in the proof of Lemma \ref{lemmfirst} shows that $\mathrm{Res}_3(f)\subset \{z\in \mathbb C, \ |z|\leq 1\}$ which then implies that $d\theta=0$. For the co-resonant state $\nu$, we can use \eqref{eq:useful} with \eqref{eq:zero} to conclude that $\mathrm{dim}(\mathrm{Res}^*_{1,\infty,\lambda_u}(f))-\mathrm{dim}(\mathrm{Res}^*_{2,\infty,\lambda_u}(f))=1$. Since Lemma \ref{lemm1} shows that  $\mathrm{dim}(\mathrm{Res}^*_{1,\infty,\lambda_u}(f))=1$, we deduce that $\mathrm{Res}^*_{2,\infty,\lambda_u}(f))=\{0\}$ and thus $d\nu =0$. This should be compared with the work of Ruelle and Sullivan on the system of Margulis measures for an Anosov diffeomorphism, see \cite[Corollary p.326]{RS}.
\end{rem}
\end{proof}
\subsection{Measure of maximal $u$-entropy which is not a measure of maximal entropy.}
In this subsection, we show Corollary \ref{cor1}.
\begin{proof}
By \cite[Theorem 20.3.9]{HK}, $\mu_{J^c_f}$ is a measure of maximal entropy if and only if $J^c_f $ is cohomologous to a constant. By the Livsic theorem \cite[Theorem 19.2.1]{HK}, this  holds if and only if for any periodic point $p$ of period $n(p)$, one has
$$\frac1{n(p)}\sum_{k=0}^{n(p)-1} J^c_f(f^kp)=\lambda_c.$$
The fact that the constant has to be equal to $\lambda_c$ follows from \cite[Corollary 2.6]{GaSh}. This holds if and only if $E_u$ and $E_s$ are jointly integrable by \cite[Theorem 5.1]{GaSh}.
\end{proof}
\section{System of Margulis measures as (co)-resonant states.}
\label{sec5}
In this section, we show that the resonant (resp. co-resonant) state at the first resonance $\lambda_u$ is given by the system of center stable (resp. strong unstable) measures. The main technical result of this section is that the system of $u$-measure is invariant with respect to the center stable holonomies.
\begin{prop}
\label{propdur}
Let $f\in\mathcal A_+^{\infty}(\mathbb T^3)$. Let $(\mu^u_x)_{x\in \mathbb T^3}$ be the system of  $u$-measures constructed in Theorem \ref{leafmeasure}. Let $R$ be a rectangle. For any $x,z\in R$, one has
\begin{equation}
\label{eq:abso}
(\mathrm{Hol}^{cs}_{x,z})_*\mu^u_z= \mu^u_x.
\end {equation}
\end{prop}
\subsection{The resonant state defines a system of center stable measures.}
We note that by Corollary \ref{cor} and Corollary \ref{lemmmuJ}, the system of $u$-measures $(\mu^u_x)_{x\in \mathbb T^3}$ coincides (up to a constant rescaling) with the conditionals $(\mu^u_{J^c_f,x})_{x\in \mathbb T^3}$ of $\mu_{J^c_f}$ along $\mathcal W^u$. 

To motivate the following construction, we recall some facts about resonances on $d_s$-forms for $d_s=\mathrm{dim}(E_s)$, see \cite[\S 4]{Hum} for a more detailed exposition\footnote{The paper treats the case of Anosov flows but one can deduce the corresponding facts for diffeomorphisms by taking suspensions.}. The first resonance for the action on $d_s$-forms is given by $e^{h_{\mathrm{top}}(f)}$ and is simple. Consider a corresponding resonant state $\alpha$ and a co-resonant state $\beta$. They satisfy by definition
\begin{equation}
\label{eq:MMEex}f_*\alpha=e^{h_{\mathrm{top}}(f)}\eta, \ \mathrm{WF}(\alpha)\subset E_{cu}^*,\quad f^*\beta=e^{ h_{\mathrm{top}}(f)}\beta, \ \mathrm{WF}(\eta)\subset E_{s}^*.
\end{equation}
One can construct the Margulis systems (corresponding to the MME) $m^s$ and $m^{cu}$ as follows. Recall that for a submanifold $\mathcal N\subset \mathbb T^3$ and a distributional current $u$, one can define the restriction of $u$ to $\mathcal N$ if $\mathrm{WF}(u)\cap N^*\mathcal N=\emptyset$, where $N^*\mathcal N:=\{(x,\xi)\in T^*\mathbb T^3\mid \forall y\in \mathcal N, \ \xi|_{T_y\mathcal N}=0\}$ is the co-normal bundle of $\mathcal N$, see \cite[Corollary 4.2.2]{Lef}. In particular, the wavefront set conditions in \eqref{eq:MMEex} allow to define
$$\forall x\in \mathbb T^3, \quad m^s_x:=\alpha|_{\mathcal W^s(x)},\quad  m^{cu}_x:=\beta|_{\mathcal W^{cu}(x)}.$$
These distributions can be shown to be measures on each leaf which scales as $e^{h_{\mathrm{top}(f)}}$ by $f$. The holonomy invariance of the systems can be shown using the absolute continuity of the foliations.

There are two major obstacles when trying to reproduce this argument in our setting for the strong unstable Margulis system. The first one is that the wavefront set condition \eqref{eq:cores} on the co-resonant state $\nu$ at the first resonance $\lambda_u$ does not allow us to define $\nu|_{\mathcal W^u(x)}$ a priori. Moreover, even if we could define it, the holonomy invariance will not follow from the usual argument since the $cs$-holonomies are not absolutely continuous.

The strategy consists in studying the resonant state $\theta$ instead and show that one can define a $cs$-Margulis system which is equal to the $cs$-conditionals of $\mu_{J^c_f}$ and then deduce the holonomy invariance of $(\mu^u_x)_{x\in \mathbb T^3}$ from a Fubini-type argument, see \S \ref{sec:hol}. Indeed, by \eqref{eq:res}, one has 
$\mathrm{WF}(\theta)\subset E_{cu}^*$ and for any $x\in \mathbb T^3$, this allows us to define $\theta_x|_{\mathcal W^{cs}(x)}$. Let $[\mathcal W^{cs}(x)]$ denote the integration current on $\mathcal W^{cs}(x)$, see for instance \cite[Example 4.1.5]{Lef}.  The distributional restriction $\theta_x$ can be computed as follows: for any $g\in C^{\infty}(\mathbb T^3)$,  one has $\theta_x(g)=(\theta, g [\mathcal W^{cs}(x)])_{\mathcal H^s_+\times \mathcal H^s_-}$. We first show that $(\theta_x)_{x\in \mathbb T^3}$ defines a system of $cs$-measures.
\begin{prop}
\label{propholu}
The system $(\theta_x)_{x\in \mathbb T^3}$ satisfies:
\begin{itemize}
\item for any $x\in \mathbb T^3$, $\theta_x$ is a non-negative measure on $\mathcal W^{cs}(x)$;
\item for any $x\in \mathbb T^3$ and $y\in \mathcal W^{cs}(x)$, one has $\theta_x=\theta_y$;
\item for any $x\in \mathbb T^3$, one has
\begin{equation}
\label{eq:scalingtheta}
f_*\theta_x=\lambda_u \theta_{f(x)};
\end{equation}
\item for any rectangle $R$ and $x,y\in R$, one has
\begin{equation}
\label{eq:holoinv}
(\mathrm{Hol}^{u}_{x,y})_*\theta_y=\theta_x.
\end{equation}
\end{itemize}
We say that $(\theta_x)_{x\in \mathbb T^3}$ is a \emph{center stable Margulis system of measures.}
\end{prop}
\begin{proof}
Let $\omega \in \Omega_2$ be a smooth $2$-form which is $C^0$-close to $d\mathrm{Vol}_{\mathcal W^{cs}}$. This means that there is a function $a\in C^{0}(\mathbb T^3)$, $C^0$-close to $1$ such that for any $g \in C^{\infty}(\mathbb T^3)$, 
$$\forall n\in \mathbb N, \quad ( f_*^n \omega, g[\mathcal W^{cs}(x)])_{\Omega_2\times \mathcal D'_1}=\int_{\mathcal W^{cs}(x)}g(y)a(f^{-n}y) d\big((f_*)^n\mathrm{vol}_{\mathcal W^{cs}(x)}\big)(y). $$
Since $\mathrm{WF}(g[\mathcal W^{cs}(x)])\subset E_{cs}^*\subset E_s^*$, one has  $g[\mathcal W^{cs}(x)]\in \mathcal H^s_-$ for some $s>0$, see \cite[Lemma 1.11]{GBGW}. Lemma \ref{lemm0} then gives, for some $c>0$,
\begin{equation}
\label{eq:thetax}
\begin{split}\theta_x(g)&=c\lim_{n\to+\infty}\lambda_u^{-n}( f_*^n \omega, g[\mathcal W^{cs}(x)])_{\mathcal H^s_+ \times \mathcal H^s_-}
\\&=c\lim_{n\to+\infty}\lambda_u^{-n}\int_{\mathcal W^{cs}(x)}g(y)a(f^{-n}y) d\big((f_*)^n\mathrm{vol}_{\mathcal W^{cs}(x)}\big)(y). 
\end{split}
\end{equation}
Since $a$ is $C^0$-close to $1$, one has $a>0$ and this shows that if $g\geq 0$, then $\theta_x(g)\geq 0$. In other words, $\theta_x$ is a non-negative measure on $\mathcal W^{cs}(x)$.
The compatibility condition is immediate from the definition of $\theta_x$. Let us show \eqref{eq:scalingtheta}. For this, we fix $g\in C^{\infty}(\mathbb T^3)$. We  use the relation $f^*(g[\mathcal W^{cs}(f(x))])=(g\circ f)[\mathcal W^{cs}(x)]$ to compute
\begin{align*}(f_*\theta_x)(g)&=\theta_x(g\circ f)=(\theta, (g\circ f)[\mathcal W^{cs}(x)])_{\mathcal H^s_+\times \mathcal H^s_-}= (\theta, f^*(g[\mathcal W^{cs}(f(x))]))_{\mathcal H^s_+\times \mathcal H^s_-}
\\&=(f_*\theta, g[\mathcal W^{cs}(f(x))])_{\mathcal H^s_+\times \mathcal H^s_-}=\lambda_u(\theta, g[\mathcal W^{cs}(f(x))])_{\mathcal H^s_+\times \mathcal H^s_-}=\lambda_u \theta_{f(x)}(g),
\end{align*}
where we used \eqref{eq:res} for the resonance $\lambda_u$. 

We now show the invariance by unstable holonomy. The important remark here is that the strong unstable foliation is absolutely continuous. This means that the standard argument gives the holonomy invariance once we have defined a $cs$-system which scales as in \eqref{eq:scalingtheta}. Note that the situation is very different from the case of the strong unstable bundle $\mu^u$ constructed in Proposition \ref{leafmeasure} because $\mathcal W^{cs}$ is \emph{not} absolutely continuous. The holonomy invariance of the $u$-system will be deduced in \S \ref{sec:hol} from the holonomy invariance of $\mu^{cs}$.

The following argument is given in \cite[Section 4.1]{BFT} but we include the proof here for completeness. Note that since $f\in \mathcal A_+^{\infty}(\mathbb T^3)$, $f^{-1}$ is a partially hyperbolic diffeomorphism with dominated splitting $T(\mathbb T^3)=E_u\oplus E_{cs}$ with $E_u$ being contracting for $f^{-1}$. We can thus use the \emph{absolute continuity} of $E_u$, see \cite[Theorem 4.3]{BFT}. Namely, there are constants $K,\beta>0$ such that for any $\epsilon>0$, for any rectangle $R$, for any $x_1,x_2\in R$ and for any two $\epsilon$-equivalent Borel sets  $A_i\subset \mathcal W^{cs}(x_i)$, $i=1,2$ (see \S \ref{sec:delta}), the measures $(\mathrm{Hol}^{u}_{x_2,x_1})_*\mathrm{vol}_{A_1}$ and $\mathrm{vol}_{A_2}$ are equivalent and
\begin{equation}
\label{eq:abs}
\left|\frac{d(\mathrm{Hol}^{u}_{x_2,x_1})_*\mathrm{vol}_{A_1}}{d\mathrm{vol}_{A_2}}-1\right|\leq K^{1+\epsilon}\epsilon^\beta.
\end{equation}
Let $R$ be a rectangle and $x,y\in R$. Let $\phi\in C^\infty(\mathcal W^{cs}(x))$ such that $\phi$ has compact support in $\mathcal W^{cs}(y)$ and let $\psi:=\phi \circ  \mathrm{Hol}^{u}_{x,y}$. Suppose that $x,y$ are close enough so that $\psi$ and $\phi$ are $\delta$-equivalent for some small $\delta>0$. By the Anosov property, there is a $\nu>0$ such that for $n\in \mathbb N$,  $\phi \circ f^n$ and $\psi  \circ f^n$ are $(e^{-\nu n}\delta)$-equivalent via $\mathrm{Hol}^{u}_{f^{-n}x,f^{-n} y}=f^{-n}\circ \mathrm{Hol}^{u}_{x,y}\circ f^n$. In the following computation, we will write $\mathrm{vol}_x$ to denote $\mathrm{vol}_{\mathcal W^{cs}(x)}$. Using \eqref{eq:abs} and \eqref{eq:lowerupper},
\begin{align*}
&\left| \int_{\mathcal W^{cs}(x)}\psi(z)a(f^{-n}z) d((f_*)^n\mathrm{vol}_{x})(z)-\int_{\mathcal W^{cs}(y)}\phi(z) a(f^{-n}z)d((f_*)^n\mathrm{vol}_{y})(z)\right|
\\&= \left| \int_{\mathcal W^{cs}(f^{-n}x)}\psi(f^nz)a(z) d\mathrm{vol}_{f^{-n}x}(z)-\int_{\mathcal W^{cs}(f^{-n}y)}\phi(f^nz)a(z) d\mathrm{vol}_{f^{-n}y}(z)\right|
\\&=\left|\int_{\mathcal W^{cs}(f^{-n}y)}\left(\phi\circ f^n \frac{d(\mathrm{Hol}^{u}_{f^{-n}y,f^{-n}x})_*\mathrm{vol}_{f^{-n}x}}{d\mathrm{vol}_{f^{-n}y}}-\phi\circ f^n\right)(z)a(z) d\mathrm{vol}_{f^{-n}y}(z)\right|
\\&=\left|\int_{\mathcal W^{cs}(f^{-n}y)}\phi\circ f^n \left(\frac{d(\mathrm{Hol}^{u}_{f^{-n}y,f^{-n}x})_*\mathrm{vol}_{f^{-n}x}}{d\mathrm{vol}_{f^{-n}y}}-1\right)(z)a(z) d\mathrm{vol}_{f^{-n}y}(z)\right|
\\&\leq K^{1+\delta e^{-\nu n}}(\delta e^{-\nu n})^{\beta}\int_{\mathcal W^{cs}(f^{-n}y)}|\phi\circ f^n(z)|a(z)d\mathrm{vol}_{f^{-n}y}(z)
\\& \leq K'\times K^{1+\delta e^{-\nu n}}(\delta e^{-\nu n})^{\beta}\lambda_u^n.
\end{align*}
Using \eqref{eq:thetax}, this shows that 
$$|(\mathrm{Hol}^{u}_{x,y})_*\theta_x(\phi)-\theta_y(\phi)|\leq \liminf_{n\to+\infty}cK'\times K^{1+\delta e^{-\nu n}}(\delta e^{-\nu n})^{\beta}=0, $$
which shows the holonomy invariance.
\end{proof}
Now, we show that $(\theta_x)_{x\in \mathbb T^3}$ identifies with the center stable conditionals $(\mu^{cs}_{J^c_f,x})_{x\in \mathbb T^3}$.
\begin{prop}
\label{prophol}
For any $x\in \mathbb T^3$,there is a $c_x>0$ such that for any $g\in C^{\infty}(\mathbb T^3)$, 
\begin{equation}
\label{eq:oh}
\lim_{n\to+\infty}((f^*)^n\theta_x)(g)=c_x\mu_{J^c_f}(g).
\end{equation}
 Moreover, for any rectangle $R$ such that $\mu_{J^c_f}(R)>0$, the conditional measures $(\mu^{cs}_{J^c_f,y})_{y\in R}$ are given $\mu_{J^c_f}$-a.e. by  $(\theta_y)_{y\in R}$, up to a constant rescaling. For $\mu_{J^c_f}$-a.e. $y\in R$, one has
\begin{equation}
\label{eq:cond}\forall z\in \mathcal W^{cs}_R(y),\quad 
(\mu^{cs}_{J^c_f,y})(z)=\frac{\theta_y(z)}{\theta_y(\mathcal W^{cs}_R(y))}.
\end{equation}

\end{prop}

\begin{proof}
Let $g\in C^{\infty}(\mathbb T^3)$ and let $x\in \mathbb T^3$. We denote by $\Pi_0^*:\mathcal H^s_-\to \mathcal H^s_-$ the spectral projector on $\mathrm{Span}(\nu)$ where $\nu$ is a co-resonant state at $\lambda_u$.  Using Lemma \ref{lemm0}, we get $c_x\in  \mathbb R$ such that $\lambda_u^{-n}(f^*)^n[\mathcal W^{cs}(x)]\to c_x\nu$ in $\mathcal H^s_-$. Lemma \ref{lemm2} then gives
\begin{align*}c_x\mu_{J^c_f}(g)&=c_x( g\theta,\nu)_{\mathcal H^s_+\times \mathcal H^s_-}=( g\theta ,\Pi_0^*[\mathcal W^{cs}(x)])_{\mathcal H^s_+\times \mathcal H^s_-} 
\\&=\lim_{n\to+\infty}\lambda_u^{-n} ( g\theta , (f^*)^n[\mathcal W^{cs}(x)])_{\mathcal H^s_+\times \mathcal H^s_-}
=\lim_{n\to+\infty}\lambda_u^{-n} ((f_*)^n( g\theta) , [\mathcal W^{cs}(x)])_{\mathcal H^s_+\times \mathcal H^s_-}
\\&=\lim_{n\to+\infty} \big(\big((f_*)^n g\big)\theta , [\mathcal W^{cs}(x)]\big)_{\mathcal H^s_+\times \mathcal H^s_-}=\lim_{n\to+\infty}\big((f^*)^n\theta_x\big)(g),
\end{align*}
where we used \eqref{eq:res} for $\lambda=\lambda_u$ and $u=\theta$.
The rest of the argument follows the proof of \cite[Lemma 8.1]{CPZ20}. We will need the following characterization of the conditional measures, see  \cite[Proposition 8.2]{CPZ20}.
\begin{lemm}
\label{lemmCPZ}
Let $\mu$ be a probability measure and let $R\subset \mathbb T^3$ be a rectangle such that $\mu(R)>0$. Let $\{\xi_n\}_{n\in \mathbb N}$ be a sequence of refining finite partitions converging to the partition $\xi$ into center-stable manifolds $(\mathcal W^{cs}_R(x))_{x\in R}$. Then there is a $R'\subset R$ such that $\mu(R)=\mu(R')$ and such that for every $y\in R'$ and every continuous $\psi:R\to \mathbb R$,
\begin{equation}
\label{eq:condi}
\int_{\mathcal W^{cs}_R(y)}\psi(z)d\mu^{cs}_y(z)=\lim_{n\to+\infty}\frac{1}{\mu(\xi_n(y))}\int_{\xi_n(y)}\psi(z)d\mu(z),
\end{equation}
where $\xi_n(y)$ is the element of $\xi_n$ containing $y$.
\end{lemm}
Let $x\in \mathbb T^3$, $n\geq 0$ and $R\subset \mathbb T^3$ be a rectangle so that $\mu_{J^c_f}(R)>0$. Let $\{\xi_n\}_{n\in \mathbb N}$ be a sequence of refining finite partitions so that for any $y\in R$, $\xi_n(y)$ is a rectangle and $\cap_{n\in \mathbb N}\xi_n(y)=\mathcal W^{cs}_R(y)$. Note that $(f^*)^n\theta_x$ is supported on $\mathcal W^{cs}(f^{-n}x)$. Fix $y\in R'$ given by Lemma \ref{lemmCPZ} and let $n,m\in \mathbb N$. Then there exists points $z^{(1)}_{n,m},\ldots, z^{(s(m))}_{n,m}\in \mathcal W^{cs}(f^{-n}x)\cap \xi_m(y)$ such that
$\mathcal W^{cs}(f^{-n}x)\cap \xi_m= \cup_{i=1}^{s(m)}\mathcal W^{cs}_R(z^{(i)}_{n,m}), $
where the union is disjoint\footnote{Since our system $(\theta_x)_{x\in \mathbb T^3}$ is defined on the whole center stable manifold $\mathcal W^{cs}(x)$ rather than the local one, any $\mathcal W^{cs}_R(z^{(i)}_{n,m})$ is fully included in $ \mathcal W^{cs}(f^{-n}x)$ in contrast to \cite[Proposition 8.2]{CPZ20}.} . We want to write $((f^*)^n\theta_x)_{|\xi_m}$ as a linear combination of measures supported on these sets. Using \eqref{eq:scalingtheta} gives, for any continuous $\psi: \xi_m(y)\to \mathbb R$, 
\begin{equation}
\begin{split}
\label{eq:11}
\int_{\xi_m(y)}\psi (z)d((f^*)^n\theta_x)(z)&=\sum_{i=1}^{s(m)}\int_{f^n\mathcal W^{cs}_R(z_{n,m}^{(i)})}\psi(f^{-n}z) d\theta_{x}(z)
\\&=\lambda_u^{-n}\sum_{i=1}^{s(m)}\int_{\mathcal W^{cs}_R(z_{n,m}^{(i)})}\psi(z) d\theta_{z_{n,m}^{(i)}}(z).
\end{split}
\end{equation}
For any $i=1,\ldots, s(m)$, using the $u$-holonomy invariance \eqref{eq:holoinv}, we obtain
\begin{align*}
 \int_{\mathcal W^{cs}_R(z_{n,m}^{(i)})}\psi(z) d\theta_{z_{n,m}^{(i)}}(z)=  \int_{\mathcal W^{cs}_R(y)}\psi(z)d\theta_{y}(z).
\end{align*}
Plugging this back into \eqref{eq:11}, this gives 
\begin{align*} \int_{\xi_m(y)}\psi(z) d((f^*)^n\theta_x)(z) =  \lambda_u^{-n}s(m) \int_{\mathcal W^{cs}_R(y)}\psi(z) d\theta_y(z).
\end{align*}
Choosing $\psi\equiv 1$, this gives using \eqref{eq:oh}
$$\lambda_u^{-n}s(m)=\frac{[(f^*)^n\theta_x](\xi_m(y))}{\theta_y(\mathcal W^{cs}_R(y))}\to_{n\to+\infty} \frac{c_x\mu_{J^c_f}(\xi_m(f))}{\theta_y(\mathcal W^{cs}_R(y))} .$$
Combining the last two equations finally yields
\begin{equation*}
 \frac{1}{\mu_{J^c_f}(\xi_m(y))}\int_{\xi_m(y)}\psi d\mu_{J^c_f} \to_{n\to+\infty}\frac{1}{\theta_y(\mathcal W^{cs}_R(y))}\int_{\mathcal W^{cs}_R(y)}\psi d\theta_y,
\end{equation*}
which, combined with \eqref{eq:condi}, gives \eqref{eq:cond}.
\end{proof}
\subsection{Holonomy invariance of $(\mu^u_x)_{x\in \mathbb T^3}$.}
\label{sec:hol}
In this subsection, we show Proposition \ref{propdur}.

\begin{proof}
Let $R$ be a rectangle such that $\mu_{J^c_f}(R)>0$. By Proposition \ref{prophol}, the $cs$-conditionals  $(\mu^{cs}_{J^c_f,x})_{x\in \mathbb T^3}$ are given, up to a constant rescalling by $(\theta_x)_{x\in \mathbb T^3}$ and are $\mathrm{Hol}^u$-invariant. Let $\theta$ be chosen so that $(\theta_x)_{x\in \mathbb T^3}=(\mu^{cs}_{J^c_f,x})_{x\in \mathbb T^3}$. For $q\in R$, recall that we defined a measure $\tilde \mu_q$ on $\mathcal W^u_R(q)$ by \eqref{eq:tildemu}.
The system of measures $(\tilde \mu_q)_{q\in R}$ is easily seen to be $\mathrm{Hol}^{cs}$-invariant
\begin{equation}
\label{eq:111}
\forall q\in R,\ \forall p\in \mathcal W^{cs}_R(q), \quad (\mathrm{Hol}^{cs}_{p,q})_*\tilde \mu_q=\tilde \mu_p.
\end{equation}
Using \eqref{eq:CPZ}, for any $\phi\in C_c^{\infty}(R,\mathbb R)$, one has
$$\mu_{J^c_f}(\phi)=\int_{\mathcal W_R^{u}(q)}\int_{\mathcal W_R^{cs}(y)}\phi(z)d\theta_y(z)d\tilde \mu^{u}_q(y). $$
We now use \eqref{eq:holoinv} and \eqref{eq:111} to obtain
\begin{align*}\mu_{J^c_f}(\phi)&=\int_{\mathcal W^u_R(q)}\int_{\mathcal W^{cs}_R(y)}\phi(z)d\theta_y(z)d\tilde \mu^{u}_q(y)=\int_{\mathcal W^u_R(q)}\int_{\mathcal W_R^{cs}(q)}\phi([z,y])d\theta_q(z)d\tilde \mu^{u}_q(y)
\\&=\int_{\mathcal W_R^{cs}(q)}\int_{\mathcal W _R^{u}(y)}\phi(z)d\tilde \mu^{u}_y(z)d\theta_q(y). 
\end{align*}
From \eqref{eq:conditional} and the last equation, we see that $\tilde \mu_q=(\mu^u_{J^c_f})_q=c_R \mu^u_q$ for some constant $c_R$ that might depend on the rectangle $R$ but not on $q\in R$. Then \eqref{eq:111} shows that $(\mu^u_q)_{q\in \mathbb T^3}$ is invariant by center stable holonomy and this concludes the proof.
\end{proof}
Note that we showed that $\mu_{J^c_f}$ has a center-stable/strong unstable local product structure. For any rectangle $R$ and $q\in R$, there is $c_R>0$ such that
\begin{equation}
\label{eq:localprod}
\forall \phi\in L^1(R,\mu_{J^c_f}),\quad \mu_{J^c_f}(\phi)=c_R\int_{\mathcal W^u_R(q)}\int_{\mathcal W^{cs}_R(y)}\phi(z)d\mu^{cs}_y(z)d \mu^{u}_q(y). 
\end{equation}
\subsection{The $u$-Margulis system as a co-resonant state.}
In this subsection, we show that the system of Margulis measures constructed in Theorem \ref{leafmeasure} coincides with the co-resonant state associated to the first resonance $\lambda_u$. Let $\varphi \in \Omega_2$ supported in a small rectangle $R$. We define a current by
\begin{equation}
\label{eq:current}
\forall q\in R, \quad \mu_q^u(\varphi):=\int_{\mathcal W_R^u(q)}\Big(\int_{\mathcal W^{cs}_R(y)}\varphi\Big)d\mu^u_q(y).
\end{equation}
\begin{lemm}
\label{lemmWF}
The current $\mu^u_q$ is independent of $q\in R$ and thus defines a current on $\mathbb T^3$ which we will denote by $\mu^u$. Moreover, one has
\begin{equation}
\label{eq:WF}
f^*\mu^u=\lambda_u \mu^u, \quad \mathrm{WF}(\mu^u)\subset E_s^*.
\end{equation}
In other words, $\mu^u$ is a co-resonant state associated to the first resonance $\lambda_u$.
\end{lemm}
\begin{proof}
Let $p,q\in R$. If $p\in \mathcal W^u_R(q)$, then $\mu^u_p=\mu^u_q$ which implies that $\mu^u_p(\varphi)=\mu^u_q(\varphi)$. If $q\in \mathcal W^{cs}_R(p)$, then using the holonomy invariance \eqref{eq:holhol}, we have
$$ \mu_q^u(\varphi)=\int_{\mathcal W^{u}_R(q)}\Big(\int_{\mathcal W^{cs}_R(y)}\varphi\Big)d\mu^u_q(y)=\int_{\mathcal W_R^u(p)}\Big(\int_{\mathcal W^{cs}_R(z)}\varphi\Big)d\big((\mathrm{Hol}^{cs}_{q,p})_*\mu^u_q\big)(z)=\mu_p^u(\varphi).$$
Using the local product structure, we deduce that $\mu^u_q$ does not depend on the point $q\in R$. Using a partition of unity $(\chi_i)_{i=1}^m$ associated to a finite cover $(R_i)_{i=1}^m$ by rectangles, we can define $\mu^u$ globally on $\mathbb T^d$:
$$\forall \varphi \in \Omega_{d_c+d_s}, \quad \mu^u(\varphi)=\sum_{i=1}^m \mu^u_{p_i}(\chi_i \varphi), \ p_i\in R_i, $$
and the value does not depend on any choice.
 The first equation in \eqref{eq:WF} is a consequence of \eqref{eq:confu}. Indeed, let $\varphi \in \Omega_{d_c+d_s}$ be such that $\varphi$ is supported in $R_i$ and $f_*\varphi$ is supported in $R_j$, then for some $p_i\in R_i$ and $p_j\in R_j$,
 \begin{align*}
 f^*\mu^u(\varphi)&=\mu^u(f_*\varphi)=\int_{\mathcal W_{R_j}^u(p_j)}\Big(\int_{\mathcal W^{cs}_{R_j}(y)}f_*\varphi\Big)d\mu^u_{p_j}(y)=\int_{\mathcal W_{R_j}^u(p_j)}\Big(\int_{f^{-1}(\mathcal W^{cs}_{R_j}(y))}\varphi\Big)d\mu^u_{p_j}(y)
 \\&=\int_{\mathcal W_{R_j}^u(p_j)}\Big(\int_{\mathcal W_{R_j}^{cs}(f^{-1}(y))}\varphi\Big)d\mu^u_{p_j}(y)=\int_{\mathcal W_{R_i}^u(p_i)}\Big(\int_{\mathcal W_{R_i}^{cs}(z)}\varphi\Big)d(f^*\mu^u_{p_j})(y)
 \\&=\lambda^u(A)\int_{\mathcal W_{R_i}^u(p_i)}\Big(\int_{\mathcal W_{R_i}^{cs}(z)}\varphi\Big)d\mu^u_{p_i}(y)=\lambda^u(A)\mu^u(\varphi).
 \end{align*} 
 It remains to prove the wavefront set bound. We note that the center-stable manifolds are a priori only $C^1$, hence, we will need to use some smoothness along the stable direction to obtain the desired regularity statement. Consider  a small rectangle $R\subset \mathbb T^3$. Let $\tau^{cs}\in \Lambda^{2}E_{cs}$ be oriented of norm $1$. Define for any Borel set $B\subset R$ the measure
 $$m_R(B):=\int_{\mathcal W^u_R(q)}\mathrm{Vol}_{\mathcal W^{cs}_R(q)}(B\cap  \mathcal W^{cs}_R(y))d\mu^u_q(y).$$
 Then we have for any $\varphi \in \Omega_{2}$ supported in  $R$:
 $$\mu^u(\varphi)=\int_R \langle \varphi, \tau^{cs}\rangle dm_R. $$
Choose a smooth transversal $\Sigma$ to $\mathcal W^s$. We first disintegrate
the Riemannian measure of each center-stable plaque along the strong stable plaques, and
then integrate the corresponding quotient measures against the measure $\mu^u_q$. We obtain a finite measure $\hat{\mu}$ on $\Sigma$ such that for any $\varphi \in \Omega_{2}$ supported in $R$, we have 
 \begin{equation}
 \label{eq:GBGW}
 \mu^u(\varphi)=\int_\Sigma\int_{\mathcal W^s_R(q)}\langle \phi(z),b_q(z)\rangle d\mathrm{Vol}_{\mathcal W^s(q)}(z)d\hat \mu(q),
 \end{equation}
 where $b_q$ is a vector-valued density.  We claim that for any $r\geq 0$, 
 \begin{equation}
 \label{eq:b_q}
 \mathrm{ess \, sup}_{q\in\Sigma}\|b_q\|_{C^r(\mathcal W^s_R(q))}<\infty.
 \end{equation}
 The proof of the claim is an adaptation of \cite[Proposition 1.8]{GBGW}. Indeed, using that $E_u$ dominates $E_{cs}$ as well as the uniform contraction of $E_s$, the backward graph transform implies that $E_{cs}$ (and hence $\tau^{cs}$) is uniformly smooth along $E_s$ in the sense above. This implies the uniform smoothness of $b_q/\|b_q\|$ along $E_s$. The norm of $b_q$ on the other hand is given by the Jacobian of strong stable holonomy between local $d_c$-dimensional transversals inside a center-stable plaque. Let $S(z)$ denote the  $d_c$-plane tangent to the center transversal inside $\Sigma\cap \mathcal W^{cs}$. Denote by $\tilde f$ the lift of $f$ to the Grassmannian, and let 
 $$\mathcal J(x,V):=|\mathrm{det}(df_x|_V)|. $$ Then \cite[Theorem 7.1,
Remark 7.2, and equation (7.3)]{Pe} (see also \cite[Equation (1.12)]{GBGW}) implies that 
 $$\mathrm{Jac}(\mathrm{Hol}^{s})(z)=\prod_{n=0}^{+\infty}\frac{\mathcal J(z,\tilde f^n S(z))}{\mathcal J(\mathrm{Hol}^sz,\tilde f^n S(\mathrm{Hol}^s z))}. $$
 Using the uniform stable contraction and the fact that $E_c$ dominates $E_s$, we deduce that the Jacobian is uniformly smooth along $E_s$, we refer to the proof of \cite[Proposition 1.8]{GBGW} for more details. 
 
 Consider $p\in R$ and  $\xi \notin E_s^*(p)$ and we shall prove that $(p,\xi)$ is not in the wavefront set. Choose a phase function $S$ such that $d_pS=\xi$ and a cutoff $\chi$ near $p$, and compute
 $$\mu^u(\chi e^{iS/h})=\int_\Sigma\int_{\mathcal W^s_R(q)}\langle \chi(z),b_q(z)\rangle e^{iS(z)/h}d\mathrm{Vol}_{\mathcal W^s(q)}(z)d\hat \mu(q). $$
 Since the density is smooth along $E_s$, we can integrate by parts in the stable variable to show that the integrand is $O(h^{\infty})$, that is, for any $N>0$, one can find a constant $C_N>0$ such that it is bounded by $C_Nh^N$, as long as $dS$ does not vanish on $\mathcal W_R^{s}(q)$, which can be ensured near $p$ by the definition of $E_{s}^*$, see \eqref{eq:dual}. This shows that $\xi \notin \mathrm{WF}(\mu^u)$, see for instance \cite[Definition 4.1.2]{Lef}, and thus $ \mathrm{WF}(\mu^u)\subset E_{s}^*$.
\end{proof}

\begin{prop}
\label{prop:last}
The conditionals of $\mu_{J^c_f}$ on the center foliation $(\mathcal W^c(x))_{x\in \mathbb T^3}$ are absolutely continuous with respect to the Lebesgue measures $d\mathrm{Vol}_{\mathcal W^c(x)}$.
\end{prop}
\begin{proof} For $z\in \mathbb T^3$ and $y\in \mathcal W^c(z)$, we define
\begin{equation}
\label{eq:densité}
\omega(z,y):=\sum_{k=0}^{+\infty}\big(J^c_f(f^{-k}y)-J^c_f(f^{-k}z)\big),
\end{equation}
where the sum converges since $\mathcal W^c$ is uniformly expanding and $J^c_f$ is Hölder continuous.
We first use that the strong unstable foliation is $C^1$ when restricted to the unstable foliation, see for instance \cite[Section 2.2.1]{ALOS}. In particular, that there is $\delta>0$, such that for any $x\in \mathbb T^3$ and $y\in \mathcal W^u(x,\delta)$, one has $(\mathrm{Hol}^u_{x,y})_*\mathrm{Vol}_{\mathcal W^c(x)}\ll\mathrm{Vol}_{\mathcal W^c(y)}$. The density can be computed using the fact that $d(f^{-k}x,f^{-k}y)\to0$ when $k\to+\infty$ and the fact that the holonomy is $C^1$. A direct computation using Jacobians shows that for any $z\in \mathbb T^3$ and $y\in \mathcal W^c(z)$, the density is given by the exponential of the function $\omega$ defined in \eqref{eq:densité}
\begin{align}
\label{eq:densi}\forall v\in \mathcal W^c(y),\quad 
\frac{d\big((\mathrm{Hol}^u_{z,y})_*\mathrm{Vol}_{\mathcal W^c(z)} \big)}{d\mathrm{Vol}_{\mathcal W^c(y)}}(v)=e^{\omega(u, \mathrm{Hol}^u_{y,z}v)}.
\end{align}
We now define a system of measures on the unstable foliation $\mathcal W^{cu}$. For any $x\in \mathbb T^3$ and any compactly supported $\phi \in C(\mathcal W^{cu}(x))$, we let
\begin{equation}
\label{eq:relat}
\mu^{cu}_x(\phi)=\int_{\mathcal W^u(x)}\left(\int_{\mathcal W^c(y)} e^{\omega(z,y)}\phi(z)d\mathrm{Vol}_{\mathcal W^c(y)}(z) \right)d\mu^u_x(y).
\end{equation}
The definition does not depend on the choice of base point $x$ in the unstable leaf by \eqref{eq:densi} and the fact that $\mu^u$ does not depend on the choice of base point in the strong unstable leaf. Since $\mu^u$ was shown to be a Borelian measure with full support on each strong unstable leaf, see Theorem \ref{theoMarg}, this shows that $\mu^{cu}$ is also a Borelian measure with full support on each unstable leaf. We now compute, using $-J^c_f(z)=\omega(f(z),z)$
\begin{align*}
\mu^{cu}_x(e^{-J^c_f(\cdot)}\phi \circ f)&=\int_{\mathcal W^u(x)}\left(\int_{\mathcal W^c(y)} e^{\omega(z,y)}e^{-J^c_f(z)}\phi(f(z))d\mathrm{Vol}_{\mathcal W^c(y)}(z) \right)d\mu^u_x(y)
\\&=\int_{\mathcal W^u(x)}\left(\int_{\mathcal W^c(y)} e^{\omega(z,y)}e^{\omega(f(z),z)}\phi(f(z))d\mathrm{Vol}_{\mathcal W^c(y)}(z) \right)d\mu^u_x(y)
\end{align*}
The cocycle relation $\omega(x,y)+\omega(y,z)=\omega(x,z)$ gives
\begin{align*}
\mu^{cu}_x(e^{-J^c_f(\cdot)}\phi \circ f)&=\int_{\mathcal W^u(x)}\left(\int_{\mathcal W^c(y)} e^{\omega(f(z),y)}\phi(f(z))d\mathrm{Vol}_{\mathcal W^c(y)}(z) \right)d\mu^u_x(y)
\\&=\int_{\mathcal W^u(x)} f^*\big(\mathrm{Vol}_{\mathcal W^c(y)}(e^{\omega(\cdot, y)}\phi(\cdot)) \big)d\mu^u_x(y)
\\&=\int_{\mathcal W^u(f(x))}\left(\int_{\mathcal W^c(y)} e^{\omega(z,y)}e^{-J^c_f(z)}\phi(f(z))d\mathrm{Vol}_{\mathcal W^c(y)}(z) \right)df_*(\mu^u_x)(y)
\\&=e^{-h^u_{\mathrm{top}}(f)}\mu_{f(x)}^{cu}(\phi)=e^{-P(J^c_f)}\mu_{f(x)}^{cu}(\phi),
\end{align*}
where we used that $h^u_{\mathrm{top}}(f)=P(J^c_f)$ by Theorem \ref{main theo} and the conformal property of $\mu^u$, see Theorem \ref{theoMarg}.

In other words, we checked that the system $(\mu^{cu}_x)_{x\in \mathbb T^3}$ is a $cu$-system which satisfies properties $(2)$ and $(3)$ of \cite[Theorem A]{CH2} for the potential $J^c_f$. We can thus apply \cite[Theorem B]{CH2} which gives that the $cu$-conditionals of the equilibrium state $\mu_{J^c_f}$ associated to $J^c_f$ are absolutely continuous with respect to~$\mu^{cu}$.

From \eqref{eq:densi}, the $c$-conditionals of $\mu^{cu}$ are absolutely continuous with respect to the Lebesgue measure on $\mathcal W^c$ and we deduce that this is also the case for $\mu_{J^c_f}$.
\end{proof}

\printbibliography[
heading=bibintoc,
title={References}]

\end{document}